\documentclass[11pt, psamsfonts]{amsart}
\usepackage{amsmath}
\usepackage{amsthm}
\usepackage{amssymb}
\usepackage{amscd}
\usepackage{amsfonts}
\usepackage{amsbsy}
\usepackage{epsfig}

\usepackage{epsfig}
\usepackage{graphicx}

\newcommand{\orb}{\text{orb}}

\newcommand{\R}{\mathbb{R}}

\def\CPM{{\mathcal C\mathcal P\mathcal M\mathcal M}}

\def\E{{\mathcal E}}

\def\K{{\mathcal K}}
\def\M{{\mathcal M}}
\def\N{{\mathcal N}}
\def\P{{\mathcal P}}
\def\Q{{\mathcal Q}}
\def\R{{\mathcal R}}

\def\htop{h_{top}}

\newcommand{\bbr}{\mbox{$\mathbb{R}$}}
\newcommand{\bbz}{\mbox{$\mathbb{Z}$}}

\newcommand{\bbn}{\mbox{$\mathbb{N}$}}

\newtheorem{Theorem}{Theorem}

\newtheorem{Proposition}{Proposition}

\newtheorem{Corollary}{Corollary}

\newtheorem*{claim*}{Claim}

\newtheorem{Definition}{Definition}
\newtheorem{Remark}{Remark}

\def\ie{{\em i.e.,\ }}

\begin{document}

\bibliographystyle{plain}
\title[Constant slope maps and the Vere-Jones classification]
{Constant slope maps and the Vere-Jones classification}
\author{Jozef Bobok and Henk Bruin}
\date{\today}
%\thanks{Both authors gratefully acknowledge the support of .....}
\address{Czech Technical University in Prague, FCE, Th{\'a}kurova 7,
166 29 Praha 6,
Czech Republic }
\email{jozef.bobok@cvut.cz}
\urladdr{http://mat.fsv.cvut.cz/bobok/}
\address{
Faculty of Mathematics, University of Vienna,
Oskar Morgensternplatz 1, 1090 Wien, Austria}
\email{henk.bruin@univie.ac.at}
\urladdr{http://www.mat.univie.ac.at/}
%$\sim$\texttt{bruin/}

\begin{abstract}
We study continuous countably piecewise monotone interval maps, and formulate conditions under which these  are conjugate to maps of constant slope,
particularly when this slope is given by the topological entropy of the map. We confine our investigation to the Markov case and phrase our conditions in the
terminology of the Vere-Jones classification of infinite matrices.
\end{abstract}

\maketitle

\section{Introduction}
For $a,b\in\bbr$, $a<b$, a continuous map $T:[a,b] \to \bbr$
is said to be piecewise monotone
if there are $k\in \bbn$ and points
$a=c_0< c_1<\cdots<c_{k-1}<c_k=b$ such that $T$ is
monotone on each $[c_i,c_{i+1}]$, $i=0,\dots,k-1$.
A piecewise monotone map $T$ has {\em constant slope} $s$ if $|T'(x)| = s$ for all $x \neq c_i$.

The following results are well known for piecewise monotone interval maps:
\begin{Theorem}\cite{Par66,MiThu88}
If $T:[0,1] \to [0,1]$ is piecewise monotone and $\htop(T)>0$ then $T$ is semiconjugate via a
continuous non-decreasing onto map $\varphi\colon~[0,1]\to [0,1]$ to a map $S$ of constant
slope $e^{\htop(T)}$. The map $\varphi$ is a conjugacy ($\varphi$ is
strictly increasing) if $T$ is transitive.
\end{Theorem}

%Moreover, for piecewise monotone interval maps  constant slope determines the topological entropy:

\begin{Theorem}\cite{MiSl80}\label{t:8} If $T$ has a constant slope $s$ then
$\htop(T)=\max\{ 0,\log s\}$.
\end{Theorem}

For continuous interval maps with a countably infinite number of pieces of
monotonicity
neither theorem is true - for examples, see \cite{MiRa05} and
\cite{BoSou11}. % see Figure~\ref{fig1}.
One of the few facts that remains true in the countably piecewise monotone setting is:

\begin{Proposition}\cite{KH95}\label{p:5}
If $T$ is $s$-Lipschitz then $\htop(T) \le \max\{ 0, \log s\}$.
\end{Proposition}

A continuous interval map $T$ has {\em constant slope} $s$ if $|T'(x)| = s$
for all but countably many points.

The question we want to address is when a {\it continuous countably piecewise monotone} interval map $T$ is conjugate to a map of constant slope $\lambda$. Particular attention will be given to the case when a slope is given by the topological entropy of $T$, which we call {\it linearizability}:

\begin{Definition}\label{d:5}
A continuous map $T:[0,1] \to [0,1]$ is said to be {\em linearizable} if it is conjugate to an {\it interval map} of constant slope $\lambda=e^{\htop(T)}$.
\end{Definition}

We will confine ourselves to the Markov case, and explore what can be said if only the transition matrix of a countably piecewise monotone map is known
in terms of the Vere-Jones classification \cite{Ver-Jo67}, refined by \cite{Rue03}.

%\iffalse
%\begin{figure}
%\unitlength=10mm
%\begin{center}
% \psfig{figure=disertace3.eps, width=5cm}
% \end{center}
% \caption{\cite{MiRa05}: For every $\varepsilon>0$ and $\lambda>2$ there exists a topologically mixing map of slope $\lambda$ and topological entropy less than $\varepsilon+\log2$.}\label{fig1}
%\end{figure}
%\fi

The structure of our paper is as follows.

In Section~\ref{s:2} ({\it $\CPM$: the class of countably piecewise monotone Markov maps}) we
make precise the conditions on continuous interval maps under which we conduct our investigation - the set of all such maps will be denoted by $\CPM$
(for {\em countably piecewise monotone Markov}). In particular, we introduce a {\em slack} countable Markov partition of a map and distinguish between {\em operator}, resp.\ {\em non-operator type}.

In Section~\ref{s:3} ({\it Conjugacy of a map from $\CPM$ to a map of constant slope}) we rephrase the key equivalence from \cite[Theorem 2.5]{Bo12}: for the sake of completeness we formulate Theorem~\ref{t:5},
which relates the existence of a conjugacy to an ``eigenvalue equation'' \eqref{e:2}, using both classical and slack countable Markov partitions,
see Definition~\ref{def:primary}.

Section~\ref{s:4} ({\it The Vere-Jones Classification}) is devoted to the Vere-Jones classification \cite{Ver-Jo67} that we use as a crucial tool in the most of our proofs in later sections.

In Section~\ref{s:5} ({\it Entropy and the Vere-Jones classification in $\CPM$}) we show in Proposition~\ref{p:7} that the topological entropy of a map in question and (the logarithm of) the Perron value of its transition matrix coincide. Using this fact, we are able to verify in Proposition~\ref{p:10} that all the transition matrices of a map corresponding to all the possible Markov partitions of that map belong to the same class in the Vere-Jones classification; so we can speak about the Vere-Jones classification of a map from $\CPM$. %In Section~\ref{s:5} we formulate our main results. %Section~\ref{s:6} is devoted to some illustrating examples.

In Section~\ref{s:6} ({\it Linearizability}) we present the main results of this text. We start with Proposition~\ref{p:3} showing two basic properties of a $\lambda$-solution of equation \eqref{e:2} and Theorem~\ref{t:7} on leo maps, see Definition~\ref{def:leo}. Afterwards we describe conditions under which a local window perturbation - Theorems~\ref{t:6},~\ref{t:9}, resp.\ a global window perturbation - Theorems~\ref{t:3}, \ref{t:10} results to a linearizable map.

In Section~\ref{s:7} ({\it Examples}) various examples illustrating linearizability/conjugacy to a map of constant slope in the Vere-Jones classes are presented.
\medskip

{\bf Acknowledgments:}
We are grateful for the support of the Austrian Science Fund (FWF): project
number P25975,
and also of the Erwin Schr\"odinger International Institute for Mathematical Physics at the University of Vienna, where the final part of this research was carried out.
Furthermore, we would like to thanks the anonymous referees for the careful
reading and valuable suggestions.

\section{$\CPM$\label{s:2}: the class of countably piecewise monotone Markov maps}

\begin{Definition}\label{def:primary}
 A countable Markov partition $\P$ for a continuous map $T\colon~[0,1]\to [0,1]$ consists of closed intervals with the following properties:
\begin{itemize}
\item Two elements of $\P$ have pairwise disjoint interiors
and $[0,1] \setminus \bigcup \P$ is at most countable.
\item The partition $\P$ is finite or countably infinite;
\item $T|_{i}$ is monotone for each $i\in \P$ ({\em classical Markov partition}) or piecewise monotone for each $i\in \P$; in the latter case we will speak
of {\em a slack Markov partition}.
\item For every $i,j\in \P$  and every maximal interval $i'\subset i$ of monotonicity of $T$, if $T(i')\cap j^{\circ}\neq \emptyset$, then $T(i') \supset j$.
\end{itemize}
\end{Definition}

\begin{Remark} The notion of a slack Markov partition will be useful in later sections of this paper where we will work with {\it window perturbations}.
If $\#\P = \infty$, then the ordinal type of $\P$ need not be ${\mathbb N}$ or $\mathbb Z$.
\end{Remark}

\begin{Definition}\label{def:CPMM}
The class $\CPM$ is the set of continuous interval maps $T:[0,1]\to [0,1]$ satisfying
\begin{itemize}
\item $T$ is {\it topologically mixing}, \ie for every open sets $U,V$ there is an $n$ such that $T^m(U)\cap V\neq\emptyset$ for all $m\ge n$.
\item $T$ admits a countably infinite Markov partition.
\item $\htop(T) < \infty$.
\end{itemize}
\end{Definition}

\begin{Remark} Since $T\in\CPM$ is topologically mixing by definition, it
cannot be constant on any subinterval of $[0,1]$.
\end{Remark}

\begin{Definition}\label{def:leo}
A map $T \in \CPM$
is called {\em leo} ({\em locally eventually onto})
if for every nonempty open set $U$ there is an $n\in{\mathbb N}$
such that $f^n(U)=[0,1]$.
\end{Definition}

\begin{Remark}\label{r:4}~
Let $T\colon~[0,1]\to [0,1]$ be a piecewise monotone Markov map, \ie such that orbits of turning points and endpoints $\{0,1\}$ are finite. Those orbits naturally determine a {\em finite} Markov partition for $T$.
This partition can be easily be refined, in infinitely ways,
to {\em countably infinite} Markov partitions.
If $T$ is topologically mixing and continuous then we will consider
$T$ as an element of $\CPM$.
\end{Remark}

\begin{Proposition}\label{p:12} Let $T\in\CPM$ with a Markov partition $\P$. For every pair $i,j\in\P$ satisfying $T(i)\supset j$ there exist a maximal $\kappa=\kappa(i,j) \in {\mathbb N}$ and intervals $i_1,\dots,i_{\kappa}\subset i$ with pairwise disjoint interiors such that $T\vert_{i_{\ell}}$ is monotone and $T(i_{\ell})\supset j$ for each $\ell=1,\dots,\kappa$.
\end{Proposition}

\begin{proof}
Since $T\in\CPM$ is topologically mixing, it is not constant on any subinterval of $[0,1]$. Fix a pair $i,j\in\P$ with $T(i)\supset j$. Since $T$ is continuous, there has to be at least one but at most a finite number of pairwise disjoint subintervals of $i$ satisfying the conclusion.
\end{proof}

For a given $T\in\CPM$ with a Markov partition $\P$, applying Proposition~\ref{p:12} we associate to $\P$ the transition
matrix $M=M(T) = (m_{ij})_{i,j\in \P}$ defined by
\begin{equation}\label{e:3}
m_{ij} = \begin{cases}
\kappa(i,j) & \text{ if } T(i) \supset j,\\
0 & \text{ otherwise.}
\end{cases}
\end{equation}
If $\P$ is a classical Markov partition of some $T\in\CPM$ then $m_{ij}\in\{0,1\}$
each $i,j \in \P$.

\begin{Remark}\label{r:5} For the sake of clarity we will write $(T,\P,M)\in\CPM^*$ when a map $T\in\CPM$, a concrete Markov partition $\P$ for $T$ and its transition matrix $M=M(T)$ with respect to $\P$ are assumed.
\end{Remark}

For an infinite matrix $M$ indexed by a countable index set $\P$ we can consider the powers $M^n=(m_{ij}(n))_{i,j\in\P}$ of $M$:
\begin{equation}\label{e:11}M^0=E=(\delta_{ij})_{i,j\in \P},
\quad M^n=\left (\sum_{k\in \P}m_{ik}m_{kj}(n-1)\right )_{i,j\in \P}, ~n\in\bbn.\end{equation}

\begin{Proposition}\label{p:11}
Let $(T,\P,M)\in\CPM^*$.
\begin{itemize}
\item[(i)] For each $n\in\bbn$ and $i,j\in\P$, the entry $m_{ij}(n)$
of $M^n$ is finite.
\item[(ii)] The entry $m_{ij}(n) = m$ if and
only if there are exactly $m$ subintervals $i_1$, $\dots$, $i_{m}$ of $i$ with
pairwise disjoint interiors such that
$T^n(i_k)\supset j$, $k=1,\dots,m$.
\end{itemize}
\end{Proposition}

\begin{proof}(i) From the continuity of $T$ and the definition of $M$ follows that the sum
$\sum_{i\in\P}m_{ij}$ is finite for each $j\in\P$, which directly
implies (i).\\
(ii) For $n=1$ this is given by the relation \eqref{e:3} defining the matrix $M$. The induction step follows from \eqref{e:11} of the product of the nonnegative matrices $M$ and $M^{n-1}$.
\end{proof}

A matrix $M$ indexed by the elements of $\P$ represents a bounded linear operator $\M$ on the Banach space $\ell^1=\ell^1(\P)$ of summable sequences indexed by $\P$, provided that
the supremum of the columnar sums is finite. Then $\M$ is realized through left multiplication
\begin{align}&\M(v):=\left (\sum_{j\in\P}m_{ij}v_j\right )_{i\in\P},~v\in\ell^1(\P),\nonumber\\&
\label{e:36}\| \M \| =
\sup_j \sum_i m_{ij}.\end{align}
The matrix $M^n$ represents the $n$th power $\M^n$ of $\M$ and by Gelfand's formula, the spectral radius $r_{\M}=\lim_{n\to\infty}\| \M^n \|^{\frac1{n}}$.

 \begin{Remark}\label{r:3}
If $(T,\P,M)\in\CPM^*$ the supremum in \eqref{e:36} is finite if and only if
 \begin{equation}\label{e:52}
\exists~K > 0 \ \forall~y\in [0,1]\colon~\#T^{-1}(y)\le K.
\end{equation}
Since this condition does not depend on a concrete choice of $\P$, we will say the map $T$ is of an {\it operator type} when the condition \eqref{e:52} is fulfilled and of a {\it non-operator type} otherwise.
\end{Remark}

\section{Conjugacy of a map from $\CPM$ to a map of constant slope}\label{s:3}

This section is devoted to the fundamental observation regarding a possible conjugacy of an element of $\CPM$ to a map of constant slope. It is presented in Theorem~\ref{t:5}.

Let $(T,\P,M)\in\CPM^*$. We are interested in positive real numbers $\lambda$ and nonzero nonnegative sequences $(v_i)_{i\in\P}$ satisfying  $Mv=\lambda v$, or equivalently

\begin{equation}\label{e:2}
\forall~i\in \P\colon~\sum_{j\in \P}m_{ij}v_j=\lambda~v_i.\end{equation}

\begin{Definition}\label{d:3}
A nonzero nonnegative sequence $v=(v_i)_{i\in\P}$ satisfying \eqref{e:2} will be called a {\em $\lambda$-solution} (for $M$). If in addition $v\in\ell^1(\P)$, it will be called a {\em summable $\lambda$-solution} (for $M$).
\end{Definition}

\begin{Remark}Since every $T\in\CPM$ is topologically mixing, any nonzero nonnegative $\lambda$-solution is in fact positive: If $v=(v_i)_{i\in\P}$ solves \eqref{e:2}, $k,j\in\P$ and $v_j>0$ then by Proposition~\ref{p:11}(ii) for some sufficiently large $n$, $\lambda^n v_k\ge m_{kj}(n)v_j>0$.\end{Remark}

Let $\CPM_{\lambda}$ denote the class of all maps from $\CPM$ of constant slope $\lambda$, \ie $S\in\CPM_{\lambda}$ if $|S'(x)| = s$ for all but countably many points.

The core of the following theorem has been proved in \cite[Theorem 2.5]{Bo12}. Since we will work with maps from $\CPM$ that are topologically mixing, we use topological conjugacies only - see \cite[Proposition 4.6.9]{ALM00}. The theorem will enable us to change freely between classical/slack Markov partitions of the map in question.

\begin{Theorem}\label{t:5}Let $T\in\CPM$. The following conditions are equivalent.
\begin{itemize}
\item[(i)] For some $\lambda >1$, the map $T$ is conjugate via a continuous increasing onto map $\psi\colon~[0,1]\to [0,1]$ to some map $S\in\CPM_{\lambda}$.
\item[(ii)] For some classical Markov partition $\P$ for $T$ there is a positive summable $\lambda$-solution $u=(u_i)_{i\in \P}$ of equation \eqref{e:2}.
    \item[(iii)] For every classical Markov partition $\P$ for $T$ there is a positive summable $\lambda$-solution $u=(u_i)_{i\in \P}$ of equation \eqref{e:2}.
    \item[(iv)] For every slack Markov partition $\Q$ for $T$ there is a positive summable $\lambda$-solution $v=(v_i)_{i\in \Q}$ of equation \eqref{e:2}.
    \item[(v)] For some slack Markov partition $\Q$ for $T$ there is a positive summable $\lambda$-solution $v=(v_i)_{i\in \Q}$ of equation \eqref{e:2}.
        \end{itemize}
\end{Theorem}

\begin{Remark} Recently, Misiurewicz and Roth \cite{MiRo16} have pointed out that if $v$ is a $\lambda$-solution of equation \eqref{e:2} that is not summable, then the map $T$ is conjugate to a map of constant slope defined on the real line or half-line.\end{Remark}

\begin{Remark}\label{r:2}Let $T\in\CPM$ be piecewise monotone with a finite Markov partition. It is well known \cite[Theorem 0.16]{Wal82}, \cite[Theorem 4.4.5]{ALM00} that the corresponding equation \eqref{e:2} has a positive  $e^{\htop(T)}$-solution which is trivially summable. By Theorem~\ref{t:5} it is also true for any countably infinite Markov partition for $T$.\end{Remark}

\begin{proof}[Proof of Theorem~\ref{t:5}]
The equivalence of (i), (ii) and (iii) has been proved in \cite{Bo12}. Since (iv) implies (iii) and (v), it suffices to show that (iii) implies (iv) and (v) implies (ii).

(iii)$\implies$(iv).~Let us assume that $\Q$ is a slack Markov partition for $T$. Obviously there is a classical partition $\P$ for $T$ which is finer than $\Q$,
\ie every element of $\P$ is contained in some element of $\Q$. Using (iii) we can consider a positive summable $\lambda$-solution $u=(u_i)_{i\in \P}$ of equation \eqref{e:2}. Let $v=(v_i)_{i\in \Q}$ be defined as

$$v_i=\sum_{i'\subset i}u_{i'};$$
  clearly the positive sequence $v=(v_i)_{i\in \Q}$ is from $\ell^1(\Q)$. Denoting $(m^{\P}_{ij})_{i,j\in\P}$ and $(m^{\Q}_{ij})_{i,j\in\Q}$ the transition matrices corresponding to the partitions $\P$, $\Q$, we can write

\begin{align}\label{a:2}
\lambda v_i&=\lambda\sum_{i'\subset i}u_{i'}=\sum_{i'\subset i}\lambda u_{i'}=\sum_{i'\subset i}\sum_{k\in\Q}\sum_{k'\subset k}m^{\P}_{i'k'}u_{k'} \\
&=\sum_{k\in\Q}\left(\sum_{k'\subset k}u_{k'}\right )\left (\sum_{i'\subset i}m^{\P}_{i'k'}\right )=\sum_{k\in\Q}m^{\Q}_{ik}v_k,\nonumber
\end{align}
where the equality $m^{\Q}_{ik}=\sum_{i'\subset i}m^{\P}_{i'k'}$ follows from the Markov property of $T$ on $\P$ and $\Q$:

\begin{equation}\label{e:24}
\text{if }T(i')\supset k'\text{ for some }k'\subset k\text{ then also } T(i')\supset k.
\end{equation}
So by \eqref{a:2}, for a given slack Markov partition $\Q$ (for $T$) we find a positive summable $\lambda$-solution $v=(v_i)_{i\in \Q}$ of equation \eqref{e:2}.

   (v)$\implies$(ii).~Assume that for some slack Markov partition $\Q$ for $T$ there is a positive summable $\lambda$-solution $v=(v_i)_{i\in \Q}$ of equation \eqref{e:2}. As in the previous part we can consider a classical Markov partition $\P$ finer than $\Q$. Using again property \eqref{e:24} let us put

   \begin{equation}\label{e:25}
u_{i'}=\sum_{T(i')\supset j}v_j,~i'\in\P.
\end{equation}
Then $u=(u_{i'})_{i'\in \P}$ is positive and we will show that it is a summable $\lambda$-solution of equation \eqref{e:2}. Fix an $i'\in\P$ and using the property \eqref{e:24} for $j\in\Q$ for which $T(i')\supset j$. Then

   \begin{equation}\label{e:26}\lambda v_j=\sum_{k\in\Q}m^{\Q}_{jk}v_k=\sum_{j'\subset j}\sum_{T(j')\supset \ell}v_{\ell}=\sum_{j'\subset j}m^{\P}_{i'j'} \sum_{T(j')\supset\ell}v_{\ell}=\sum_{j'\subset j}m^{\P}_{i'j'} u_{j'},
   \end{equation}
   hence summing \eqref{e:26} through all $j$'s from $\Q$ that are $T$-covered by $i'\in\P$, we obtain with the help of \eqref{e:25},

      \begin{equation*}\label{e:27}\lambda u_{i'}=\sum_{T(i')\supset j}\lambda v_j=\sum_{T(i')\supset j}\sum_{j'\subset j}m^{\P}_{i'j'} u_{j'}=\sum_{j'\in\P}m^{\P}_{i'j'} u_{j'}.
   \end{equation*}
Since by our assumption on $v=(v_i)_{i\in \Q}$ and \eqref{e:25}

\begin{equation*}\label{e:37}
\sum_{i'\in\P}u_{i'}=\sum_{i\in\Q}\sum_{j\in\Q}m^{\Q}_{ij}v_j=
\lambda\sum_{i\in\Q} v_{i}<\infty,
\end{equation*}
so $u=(u_{i'})_{i'\in \P}$ is a summable $\lambda$-solution of equation \eqref{e:2}.
\end{proof}

Maps $T\in\CPM$ are continuous, topologically mixing with positive topological entropy.
Thus all possible semiconjugacies described in \cite[Theorem 2.5]{Bo12} will be in fact conjugacies, see \cite[Proposition 4.6.9]{ALM00}.
Many properties hold under the assumption of positive entropy or for {\em countably piecewise continuous} maps.
One interesting example of a countably piecewise continuous and countably piecewise monotone (still topologically mixing) map will be presented in Section~\ref{s:7}.
However, since the technical details are much more involved and would obscure
the ideas, we confine the proofs to $\CPM$.

\section{The Vere-Jones Classification}\label{s:4}

Let us consider a matrix $M=(m_{ij})_{i,j\in\P}$, where the index set $\P$ is finite or countably infinite. The matrix $M$ will be called
\begin{itemize}\item {\em irreducible},
if for each pair of indices $i,j$ there exists a positive integer $n$ such
that $m_{ij}(n)>0$, and
\item {\em aperiodic}, if for each index $i\in\P$ the value $gcd\{\ell\colon~m_{ii}(\ell)>0\}=1$.
\end{itemize}

\begin{Remark}
Since $T\in\CPM$ is topologically mixing, its transition matrix $M$ is irreducible and aperiodic.
\end{Remark}

In the sequel we follow the approach suggested by Vere-Jones \cite{Ver-Jo67}.

\begin{Proposition}\label{p:2}(i)~Let $M=(m_{ij})_{i,j\in\P}$ be
a nonnegative irreducible aperiodic matrix indexed by a countable index set $\P$. There exists a
common value $\lambda_M$ such that for each $i,j$
\begin{equation}\label{e:13}\lim_{n\to\infty} [m_{ij}(n)]^{\frac{1}{n}}=\sup_{n\in{\tiny \bbn}}[m_{ii}(n)]^{\frac{1}{n}}=\lambda_M.\end{equation}

(ii)~For any value $r>0$ and all $i,j$
\begin{itemize}
\item the series $\sum_{n}m_{ij}(n)r^n$ are either all convergent or all divergent;
\item as $n\to\infty$, either all or none of the sequences
$\{m_{ij}(n)r^n\}_{n}$ tend to zero.
\end{itemize}
\end{Proposition}

\begin{Remark}\label{r:6}The number $\lambda_M$ defined by \eqref{e:13} is often called the {\it Perron value of $M$}. In the whole text we will assume that for a given nonnegative irreducible aperiodic matrix $M=(m_{ij})_{i,j\in\P}$ its Perron value $\lambda_M$ is finite.\end{Remark}

\subsection{Entropy, generating functions and the Vere-Jones classes}\label{ss:15}

To a given  irreducible aperiodic matrix $M=(m_{ij})_{i,j\in
\P}$ with entries from $\bbn\cup\{0\}$ corresponds a strongly connected directed graph $G=G(M)=(\P,\E\subset \P\times\P)$ containing $m_{ij}$ edges from $i$ to $j$.

 The {\em Gurevich entropy} of $M$ (or of $G=G(M)$) is defined as
\begin{equation*}\label{e:12}
h(G)=h(M) = \sup\{\log r(M') : M' \text{ is a finite submatrix of } M \},
\end{equation*}
where $r(M')$ is the large eigenvalue of the finite transition matrix $M'$.
Gurevich proved that

\begin{Proposition}\cite{Gur69}\label{p:4}
$h(M) = \log \lambda_M$.
\end{Proposition}

Since by Proposition~\ref{p:2} the value $R=\lambda_M^{-1}$ is a common radius of convergence of the power series $M_{ij}(z)=\sum_{n\ge 0}m_{ij}(n)z^n$, we immediately obtain for each pair $i,j\in\P$,

\begin{equation*}\label{e:10}M_{ij}(r)\begin{cases}
\in\bbr, & 0\le r<R,\\
=\infty, & r>R.
\end{cases}
\end{equation*}

It is well known that in $G(M)$
\begin{itemize}
\item $m_{ij}(n)$ equals to the number of paths of length $n$ connecting $i$ to $j$.
\end{itemize}
 Following \cite{Ver-Jo67}, for each $n\in\bbn$ we will consider the following coefficients:
\begin{itemize}
\item First entrance to $j$: $f_{ij}(n)$ equals  the number of paths of length $n$ connecting $i$ to $j$, without appearance of $j$ in between.
    \item Last exit of $i$: $\ell_{ij}(n)$ equals the number of paths of length $n$ connecting $i$ to $j$, without appearance of $i$ in between.

\end{itemize}
 Clearly $f_{ii}(n)=\ell_{ii}(n)$ for each $i\in\P$. Also it will be useful to introduce
\begin{itemize}

        \item First entrance to $\P'\subset \P$: for a nonempty $\P'\subset\P$ and $j\in\P'$, $g^{\P'}_{ij}(n)$ equals the number of paths of length $n$ connecting $i$ to $j$, without appearance of any element of $\P'$ in between.
\end{itemize}

  The first entrance to $\P'\subset \P$ will provide us a new type of a generating function used in \eqref{a:500} and its applications.

\begin{Remark}\label{r:1}Let us denote by $\Phi_{ij}$, $\Lambda_{ij}$ the radius of convergence of the power series $F_{ij}(z) = \sum_{n\ge 1} f_{ij}(n) z^n$, $L_{ij}(z) = \sum_{n\ge 1} \ell_{ij}(n) z^n$. Since $f_{ij}(n)\le m_{ij}(n)$, $\ell_{ij}(n)\le m_{ij}(n)$ for each $n\in\bbn$ and each $i,j\in\P$, we always have $R\le \Phi_{ij}$, $R\le \Lambda_{ij}$.
\end{Remark}

Proposition~\ref{p:1} has been stated in \cite{Rue03}. Since the argument showing the part (i) presented in \cite{Rue03} is not correct, we offer our own version of its proof.

\begin{Proposition}\label{p:1}\cite[Proposition 2.6]{Rue03} Let $(T,\P,M)\in\CPM^*$, consider the graph $G=G(M)$, $R=\lambda_M^{-1}$.
\begin{itemize}
\item[(i)] If there is a vertex $j$ such that $R=\Phi_{jj}$ then there exists a strongly connected subgraph $G'\subsetneq G$ such that $h(G')=h(G)$.
    \item[(ii)] If there is a vertex $j$ such that $R<\Phi_{jj}$ then for all proper strongly connected subgraphs $G'$ one has $h(G')<h(G)$.
        \item[(iii)] If there is a vertex $j$ such that $R<\Phi_{jj}$ then $R<\Phi_{ii}$ for all $i$.
    \end{itemize}
 \end{Proposition}
 \begin{proof}For the proof of part (ii) see \cite{Rue03}.

 Let us prove (i).
  Fix a vertex $j\in\P$ for which $R=\Phi_{jj}$ and choose arbitrary $i\neq j$. We can write

\begin{equation}\label{e:51}f_{jj}(n)=~_if_{jj}(n)+ ~^if_{jj}(n),\end{equation}
where $_if_{jj}(n)$, resp.\ $^if_{jj}(n)$ denotes the number of $f_{jj}$-paths of length $n$ that do not contain $i$, resp.\ contain $i$.\newline
\noindent {\bf I.} If $\limsup_{n\to\infty}[_if_{jj}(n)]^{1/n}=\lambda_M$,
then there is nothing to prove. \newline
\noindent {\bf II.} Assume that $\limsup_{n\to\infty}[_if_{jj}(n)]^{1/n}<\lambda_M$. Then by our assumption and \eqref{e:51}

\begin{equation}\label{e:49}\limsup_{n\to\infty}[^if_{jj}(n)]^{1/n}=\lambda_M.
\end{equation}
Let us denote $g_{ij}(n)$ the number of paths of length $n$ connecting $i$ to $j$, without appearance of $i,j$ after the initial $i$ and before the final $j$. If we denote $^{1,j}f_{ii}(n)$ the number of $f_{ii}$-paths of length $n$ connecting $i$ to $i$ with exactly one appearance of $j$ after the initial $i$ and before the final $i$, we can write for $n\ge 2$ (the coefficients $_jm_{ii}(n)$ are defined analogously as $_jf_{ii}(n)$ - compare the proof of Theorem \ref{t:12})

\begin{align}\label{e:50}^if_{jj}(n) &=\sum_{m=2}^n\sum_{k=1}^{m-1}g_{ji}(k)~_jm_{ii}(n-m)g_{ij}(m-k) \\
&=\sum_{m=2}^n {_j}m_{ii}(n-m)\sum_{k=1}^{m-1}g_{ji}(k)~g_{ij}(m-k)\nonumber\\
&=\sum_{m=2}^n {_j}m_{ii}(n-m)~ ^{1,i}f_{jj}(m)
=\sum_{m=2}^n {_j}m_{ii}(n-m)~ ^{1,j}f_{ii}(m)\nonumber.
\end{align}
By the formula of \cite[Lemma 4.3.6]{ALM00} and our assumption \eqref{e:49}, for arbitrary $i\in\P\setminus\{j\}$ we obtain from \eqref{e:50} either

\begin{equation}\label{e:5}\limsup_{n}[{_j}m_{ii}(n)]^{1/n}=\lambda_M\end{equation}
or
\begin{equation}\label{e:6}\limsup_{n\to\infty}~[^{1,j}f_{ii}(n)]^{1/n}=\lambda_M.
\end{equation}
If \eqref{e:5} is fulfilled for some $i$ the existence of a strongly connected subgraph $G'\subsetneq G$ such that $h(G')=h(G)$ immediately follows. Otherwise, since

$$\limsup_{n\to\infty}~[^{1,j}f_{ii}(n)]^{1/n}\le \limsup_{n\to\infty}~[f_{ii}(n)]^{1/n},$$
we get  $R=\Phi_{ii}$ for each $i\in\P$ and the conclusion follows from \cite[Theorem 2.2]{Sal88}.
The assertion (iii) immediately follows from (i) and (ii).
\end{proof}

The behavior of the series $M_{ij}(z)$, $F_{ij}(z)$ for $z=R$ was used in the Vere-Jones classification of irreducible aperiodic matrices \cite{Ver-Jo67}. Vere-Jones originally distinguished $R$-{\it transient}, {\it null $R$-recurrent} and {\it positive $R$-recurrent} case. Later on, the classification was refined by Ruette in \cite{Rue03}, who added {\it strongly positive $R$-recurrent} case.  All is summarized in Table 1 which applies independently of the sites $i, j \in \P$ for $M$ irreducible - compare the last row of Table 1 and Proposition~\ref{p:1}. We call corresponding classes of matrices {\it transient, null recurrent, weakly recurrent, strongly recurrent}. The last three, resp.\ two  possibilities will occasionally be summarized by '$M$ is recurrent', resp.\ '$M$ is positive recurrent'.
\begin{quote}
\vline
\begin{tabular}{c|c|c|c|c|c|c|c|c|c|c|c|c|c}

\hline
& transient & null & weakly & strongly \\
& & recurrent & recurrent & recurrent\\
\hline
$F_{ii}(R)$ & $< 1$ & $= 1$ & $= 1$ & $= 1$ \\[2mm]
\hline
$F'_{ii}(R)$ & $\le \infty$ & $\infty$ & $< \infty$
 & $< \infty$ \\[2mm]
 \hline
$M_{ij}(R)$ & $< \infty$ & $= \infty$ & $= \infty$
 & $= \infty$ \\[2mm]
 \hline
$\lim_{n \to\infty} m_{ij}(n) R^n$ & $=0$ & $=0$  & $\lambda_{ij}\in (0,\infty)$
& $\lambda_{ij}\in (0,\infty)$ \\[2mm]
\hline
for all $i$ & $R = \Phi_{ii}$ & $R=\Phi_{ii}$ & $R = \Phi_{ii}$ & $R < \Phi_{ii}$\\[2mm]
\hline
\end{tabular}
\vline
\vskip2mm
\centerline{\bf Table 1.}
\end{quote}

\subsubsection{Salama's criteria}
There are geometrical criteria - see \cite{Sal88} and also \cite{Rue03} - for cases of the Vere-Jones classification to apply depending on whether the underlying strongly connected directed graph can be enlarged/reduced (in the class of strongly connected directed graphs) without changing the entropy. We will use some of them in Section~\ref{s:7}.

\begin{Theorem}\label{t:1}\cite{Sal88,Rue03}~The following are true:
\begin{itemize}
\item[(i)] A graph $G$ is transient if and only if there is a graph $G'$ such that $G\subsetneq G'$ and $h(G) = h(G')$.
    \item[(ii)] $G$ is strongly recurrent if and only if $h(G_0) < h(G)$ for any $G_0\subsetneq G$.
    \item[(iii)] $G$ is recurrent but not strongly recurrent if and only if there exists $G_0\subsetneq G$ with $h(G_0) = h(G)$, but $h(G) < h(G_1)$ for every $G_1\supsetneq G$.

\end{itemize}
\end{Theorem}

\subsubsection{Further useful facts}

 In the whole paper we are interested in nonzero nonnegative solutions of equation \eqref{e:2}. Analogously, in the next proposition we consider nonzero nonnegative {\it subinvariant} $\lambda$-solutions $v=(v_i)_{i\in\P}$ for a matrix $M$,
\ie satisfying the inequality $Mv\le \lambda v$.

\begin{Theorem}\label{t:2}\cite[Theorem 4.1]{Ver-Jo67} Let $M=(m_{ij})_{i,j\in\P}$ be irreducible. There is no subinvariant $\lambda$-solution for $\lambda<\lambda_M$. If $M$ is transient there are infinitely many linearly independent subinvariant $\lambda_M$-solutions. If $M$ is recurrent there is a unique subinvariant $\lambda_M$-solution which is in fact $\lambda_M$-solution of equation \eqref{e:2} proportional to the vector $(F_{ij}(R))_{i\in\P}$ ($j\in\P$ fixed), $R=\lambda_M^{-1}$. \end{Theorem}

A general statement (a slight adaption of \cite[Theorem 2]{Pr64}) on solvability of equation \eqref{e:2} is as follows:
\begin{Theorem}\label{t:12}~Let $M=(m_{ij})_{i,j\in\P}$ be  irreducible. The system $Mv=\lambda v$ has a nonzero nonnegative solution $v$ if and only if
\begin{itemize}
\item[(a)]$\lambda=\lambda_M$ and $M$ is recurrent, or
\item[(b)] when either $\lambda>\lambda_M$ or\newline
$\lambda=\lambda_M$ and $M$ is transient,\newline
 there is an infinite sequence of indices $K\subset \P$ such that ($z=\lambda^{-1}$)
    \begin{equation}\label{e:15}\lim_{j\to\infty}\lim_{k\to\infty,~k\in K}\frac{\sum_{\alpha=j}^{\infty}m_{i\alpha}~ M_{\alpha k}(z)}{\sum_{\alpha=1}^{\infty}m_{i\alpha}~ M_{\alpha k}(z)}=0
    \end{equation}
    \end{itemize}
    for each $i\in\P$.
\end{Theorem}
\begin{proof}
Following Chung \cite{Chu60} we will use the analogues of the taboo probabilities: For $k\in\P$ define $_km_{ij}(1)=m_{ij}$ and for $n\ge 1$,
        $$_km_{ij}(n+1)=\sum_{\alpha\neq k}m_{i\alpha}~_km_{\alpha j}(n);$$
clearly, $_km_{ij}(n)$ equals to the number of paths of length $n$ connecting $i$ to $j$ with no appearance of $k$ between. Denote also
$^km_{ij}(n)=m_{ij}(n)-_km_{ij}(n)$ the number of paths of length $n$ connecting $i$ to $j$ with at least one appearance of $k$ between.
        The usual convention that $_km_{ij}(0)=\delta_{ij}(1-\delta_{ik})$ will be used. The following identities directly follow from the definitions of the corresponding generating functions - see before Table 1 - or are easy to verify:~For all $i,j,k\in\P$ and $0\le z<R$,
\begin{itemize}
\item[(i)] $M_{ik}(z)=_jM_{ik}(z)+^jM_{ik}(z)$,
\item[(ii)] $_iM_{ik}(z)=L_{ik}(z)$,
\item[(iii)] $^jM_{ik}(z)=M_{ij}(z)L_{jk}(z)$,
\item[(iv)] $M_{ik}(z)=M_{ii}(z)L_{ik}(z)$,
\item[(v)] \cite{Ver-Jo67} for $i\neq k$,
\begin{equation*}\frac{\sum_{\alpha\le j-1}m_{i\alpha}M_{\alpha k}(z)}{M_{ik}(z)}+\frac{\sum_{\alpha\ge j}m_{i\alpha}M_{\alpha k}(z)}{M_{ik}(z)}=\frac{1}{z},\end{equation*}
\item[(vi)] \cite{Ver-Jo67} $\sum_{\alpha\ge 1}m_{i\alpha}M_{\alpha i}(z)=\frac{M_{ii}(z)}{z}-\frac{1}{z}$ is finite.
\end{itemize}
By \cite[Theorem 2]{Pr64} the double limit \eqref{e:15} can be replaced by
\begin{equation}\label{e:18}\lim_{j\to\infty}\lim_{k\to\infty,~k\in K}\frac{\sum_{\alpha=j}^{\infty}m_{i\alpha}~ _iM_{\alpha k}(1/\lambda)}{_iM_{ik}(1/\lambda)}=0.
\end{equation}

Using the identities (i)-(vi) we can write \eqref{e:18} as

\begin{align*}\label{a:1}
A(j,k) &:=\frac{\sum_{\alpha=j}^{\infty}m_{i\alpha}~ _iM_{\alpha k}(z)}{_iM_{ik}(z)}=\frac{\sum_{\alpha=j}^{\infty}m_{i\alpha}~ M_{\alpha k}(z)}{L_{ik}(z)}-\frac{\sum_{\alpha=j}^{\infty}m_{i\alpha}~ ^iM_{\alpha k}(z)}{L_{ik}(z)}\\
&=M_{ii}(z)\frac{\sum_{\alpha=j}^{\infty}m_{i\alpha}~ M_{\alpha k}(z)}{M_{ii}(z)L_{ik}(z)}-\frac{\sum_{\alpha=j}^{\infty}m_{i\alpha}~ M_{\alpha i}(z)L_{ik}(z)}{L_{ik}(z)}\\
&=M_{ii}(z)\frac{\sum_{\alpha=j}^{\infty}m_{i\alpha}~ M_{\alpha k}(z)}{M_{ik}(z)}-\sum_{\alpha=j}^{\infty}m_{i\alpha}~ M_{\alpha i}(z)=:B(j,k).
\end{align*}
Since by (vi), $\lim_{j\to\infty}\sum_{\alpha=j}^{\infty}m_{i\alpha}~ M_{\alpha i}(z)=0$,
using (v) we obtain that

$$\lim_{j\to\infty}\lim_{k\to\infty,~k\in K}A(j,k)=0$$
if and only if

$$
\lim_{j\to\infty}\lim_{k\to\infty,~k\in K}B(j,k)=\lim_{j\to\infty}\lim_{k\to\infty,~k\in K}\frac{\sum_{\alpha=j}^{\infty}m_{i\alpha}~ M_{\alpha k}(z)}
{\sum_{\alpha=1}^{\infty}m_{i\alpha}~ M_{\alpha k}(z)}=0.
$$
The conclusion follows from \cite[Theorem 2]{Pr64}.
\end{proof}

\begin{Corollary}\label{c:1}If for each $i$, $m_{ij}=0$ except for a finite set of $j$ values, then $Mv=\lambda v$ has a nonzero nonnegative solution if and only if $\lambda\ge \lambda_M$.
\end{Corollary}

\subsubsection{Useful matrix operations in the Vere-Jones classes}

In order to be able to modify nonnegative matrices in question we will need the following observation. In some cases it will enable us to produce transition matrices of maps from $\CPM$.  Let $E$ be the identity matrix, see
\eqref{e:11}.

\begin{Proposition}\label{p:19}Let $M=(m_{ij})_{i,j\in\P}$ be  irreducible.  For arbitrary pair of positive integer $k$ and nonnegative integer $\ell$ consider the matrix $N=kM+\ell E$. Then
\begin{itemize}
\item[(i)] $\lambda_N=k\lambda_M+\ell$,
\item[(ii)] if for each $i$, $m_{ij}=0$ except for a finite set of $j$ values, the matrix $N$ belongs to the same class of the Vere-Jones classification as the matrix $M$.
\end{itemize}
\end{Proposition}
\begin{proof}Both the conclusions clearly hold if $N$ is a multiple of $M$, \ie when $\ell=0$. So to show our statement it is sufficient to verify the case when $N=M+E$.

(i) Since $Mv=\lambda v$ if and only if $Nv=(\lambda+1)v$, property (i) follows from Corollary~\ref{c:1}.

(ii) By our assumption, for each $i$, $n_{ij}=0$ except for a finite set of $j$ values, so Theorem~\ref{t:12} and Corollary~\ref{c:1} can be applied. Notice that for any nonnegative $v$, \begin{equation}\label{e:47}
Mv\le\lambda v\text{ if and only if }Nv\le(\lambda+1)v,\end{equation}
so by Theorem~\ref{t:2}, the matrix $M$ is transient, resp.\ recurrent if and only if $N=M+E$ is transient, resp.\ recurrent. In order to distinguish different recurrent cases we will use Table 1. Since by (i) $\lambda_N=\lambda_M+1$, we can write

\begin{equation}\label{a:1}
\frac{n_{11}(n)}{\lambda_N^n} =
\frac{\sum_{k=0}^{n}\binom{n}{k}m_{11}(k)}
{\sum_{k=0}^{n}\binom{n}{k}\lambda_M^k}
= \underbrace{\frac{\sum_{k=0}^{n_1-1}\binom{n}{k}\frac{m_{11}(k)}{\lambda_M^k}\lambda_M^k}
{\sum_{k=0}^{n}\binom{n}{k}\lambda_M^k}}_{U(n,n_1)}
+
\underbrace{\frac{\sum_{k=n_1}^{n}\binom{n}{k}\frac{m_{11}(k)}{\lambda_M^k}\lambda_M^k}
{\sum_{k=0}^{n}\binom{n}{k}\lambda_M^k}}_{V(n,n_1)}.
\end{equation}
By \eqref{e:13} $\frac{m_{11}(k)}{\lambda_M^k}\le 1$ for each $k$ and we can put $\mu=\lim_{k\to\infty}\frac{m_{11}(k)}{\lambda_M^k}$. For each $\varepsilon>0$ there exists $n_1 \in {\mathbb N}$ such that $\frac{m_{11}(k)}{\lambda_M^k}\in (\mu-\varepsilon,\mu+\varepsilon)$ whenever $k>n_1$. Then using the fact that

$$
\lim_{n\to\infty}U(n,n_1)=0,
\qquad \lim_{n\to\infty}\frac{\sum_{k=n_1}^{n}\binom{n}{k}\lambda_M^k}{\sum_{k=0}^{n}\binom{n}{k}\lambda_M^k}=1,
$$ we can write for any $\delta>0$ and sufficiently large $n=n(\delta)$:

\begin{equation*}\mu-\varepsilon\le \frac{n_{11}(n)}{\lambda_N^n}\le \delta+\mu+\varepsilon,\end{equation*}
hence $\lim_n\frac{n_{11}(n)}{\lambda_N^n}=\lim_n\frac{m_{11}(n)}{\lambda_M^n}=\mu$. By Table 1, $M$ is null, resp.\ positive recurrent if and only if $N$ is null, resp.\ positive recurrent.

Finally, let $M=(m_{ij})_{i,j\in\P}$ be positive recurrent and assume its irreducible submatrix $K=(k_{ij})_{i,j\in\P'}$ for some $\P'\subset \P$, denote $L=K+E$. Then similarly as above we obtain that $\lambda_N=\lambda_M+1$, resp.\ $\lambda_L=\lambda_K+1$. If $M$ is weakly, resp.\ strongly recurrent, then for some $K$, resp.\ for each $K$ we obtain $\lambda_N=\lambda_L$, resp.\ $\lambda_N>\lambda_L$ and Theorem~\ref{t:1} can be applied.

This finishes the proof for $N=M+E$. Now, the case when $N=M+\ell E$, $\ell>1$, can be verified inductively.
\end{proof}

\section{Entropy and the Vere-Jones classification in $\CPM$}\label{s:5}

 The following statement identifies the topological entropy of a map and the Perron value $\lambda_M$ of its transition matrix.

\begin{Proposition}\label{p:7} Let $(T,\P,M)\in\CPM^*$. Then $\lambda_M= e^{\htop(T)}$.
and if there is a summable $\lambda$-solution of equation \eqref{e:2} then $\lambda_M\le \lambda$.
\end{Proposition}
\begin{proof}
For the first equality, we start by proving $\lambda_M\le e^{\htop(T)}$. We use Proposition~\ref{p:2}(i) and Proposition~\ref{p:11}(ii). By those
statements, $\lambda_M=\lim_n [m_{jj}(n)]^{\frac{1}{n}}$ for any $j\in\P$ and and for each sufficiently large $n$, the interval $j$
contains $m_{jj}(n)$ intervals $j_1,\dots,j_{m_{jj}(n)}$ with
pairwise disjoint interiors such that $T^n(j_i)\supset j_{\ell}$ for all $1\le i,\ell\le
m_{jj}(n)$. Clearly, the map $T^n$ has a $m_{jj}(n)$-horseshoe \cite{Mi79} hence $\htop(T^n)=n\htop(T)\ge \log m_{jj}(n)$ and $e^{\htop(T)}\ge
[m_{jj}(n)]^{\frac{1}{n}}$. Since $n$ can be arbitrarily large, the inequality
$\lambda_M\le e^{\htop(T)}$ follows.

Now we look at the reverse inequality
$\lambda_M\ge e^{\htop(T)}$. A pair $(S,T\vert_{S})$ is a subsystem of $T$ if $S\subset [0,1]$ is closed and $T(S)\subset S$. It has been showed in \cite[Theorem 3.1]{Bo03} that the entropy of $T$ can be expressed as the supremum of entropies of {\em minimal} subsystems. Let us fix a minimal subsystem $(S(\varepsilon),T\vert_{S(\varepsilon)})$ of $T$ for which $\htop(T\vert_{S(\varepsilon)})>\htop(T)-\varepsilon>0$.

\noindent {\bf Claim.}
{\it There are finitely many elements $i_1,\dots,i_m\in\P$ such that $S(\varepsilon)\subset\bigcup_{j=1}^mi_j^{\circ}$.}

\noindent {\it Proof of Claim.}~Let us denote $P=[0,1]\setminus\bigcup_{i\in\P}i^{\circ}$. Then $P$ is closed, at most countable and $T(P)\subset P$. Assume that $x\in P\cap S(\varepsilon)\neq\emptyset$. Then $\orb_T(x)\subset P$ which is impossible for $(S(\varepsilon),T\vert_{S(\varepsilon)})$ minimal of positive topological entropy. If $S(\varepsilon)$ intersected infinitely many elements of $\P$ then, since $S(\varepsilon)$ is closed, it would intersect also $P$, a contradiction. Thus, there are finitely many $i_1,\dots,i_m \in \P$ of the required property.\hskip110mm $\blacksquare$
\\[3mm]
Our claim together with Proposition~\ref{p:12} say that connect-the-dots map of $(S(\varepsilon),T\vert_{S(\varepsilon)})$ is piecewise monotone and the finite submatrix $M'$ of $M$ corresponding to the elements $i_1,\dots,i_m$ satisfies $r(M')\ge e^{\htop(T\vert_{S(\varepsilon)})}$. Now the conclusion follows from Proposition~\ref{p:4}.

The second statement follows from Theorem~\ref{t:5}, Proposition~\ref{p:5} and the fact that topological entropy is a conjugacy invariant: $\lambda_M=e^{\htop(T)}=e^{\htop(S)}\le \lambda$.
\end{proof}

We would like to transfer the Vere-Jones classification to $\CPM$. That is why it is necessary to be sure that a change of Markov partition for the map in question does not change the Vere-Jones type of its transition matrix. This is guaranteed by the following proposition.

Given $T \in \CPM$, consider the family $(\P_{\alpha})_{\alpha}$ of {\it all Markov partitions} for $T$. Write $Q_{\alpha}=[0,1]\setminus\bigcup_{i\in \P_{\alpha}}i^{\circ}$. The minimal Markov partition $\R$ for $T$ consists of the closures of connected components of $[0,1]\setminus\bigcap_{\alpha}Q_{\alpha}$.

\begin{Proposition}\label{p:10}Let $T\in\CPM$ with two Markov partitions $\P$, resp.\ $\Q$ and corresponding matrices $M^{\P}=(m^{\P}_{ij})_{i,j\in\P}$, resp.\ $M^{\Q}=(m^{\Q}_{ij})_{i,j\in\Q}$. The matrices $M^{\P}$ and $M^{\Q}$ belong to the same class of the Vere-Jones classification.\end{Proposition}
\begin{proof}Since the map $T$ is topologically mixing, each of the matrices $M^{\P}$, $M^{\Q}$ is irreducible and aperiodic. Moreover, by Proposition~\ref{p:7} the value guaranteed in Proposition~\ref{p:2}(i) equals $e^{\htop(T)}$ and so is the same for both the matrices $M^{\P}$, $M^{\Q}$ - denote it $\lambda$.
Let $P=[0,1]\setminus\bigcup_{j\in\P}j^{\circ}$ and $Q=[0,1]\setminus\bigcup_{j\in\Q}j^{\circ}$.

First, let us assume that $P\subset Q$. Fix two elements $j\in\P$, resp.\ $j'\in\Q$ such that $j'\subset j$. Let us consider a path of the length $n$

\begin{equation}\label{e:39}j=j_0\rightarrow^{\P}j_1
\rightarrow^{\P}j_2\cdots\rightarrow^{\P}j_n=j\end{equation}
with respect to $\P$; by Proposition~\ref{p:12} each interval $j_i$ contains $k_i=k^{\P}(j_i,j_{i+1})$ intervals of monotonicity of $T$ - denote them $\iota_i(1),\dots,\iota_i(k_i)$ - such that $T(x)\notin j_{i+1}$ whenever $x\in j_i\setminus\bigcup_{\ell=1}^{k_i}\iota_i$. This implies that

\begin{equation}\label{e:38}\prod_{i=0}^{n-1}k_i
\end{equation}
is the number of paths with respect to $\P$ through the same vertices in order given by \eqref{e:39} and, at the same time, it is an upper bound of a number of paths

$$
j'=j'_0\rightarrow^{\Q}j'_1\rightarrow^{\Q}j'_2\cdots\rightarrow^{\Q}j'_n=j'
$$
with respect to finer $\Q$ such that $j'_i\subset j_i$ for each $i$. Considering all possible paths in \eqref{e:39} and summing their numbers given by \eqref{e:38}, we obtain

\begin{equation}\label{e:19}
m^{\Q}_{j'j'}(n)\le m^{\P}_{jj}(n)
\end{equation}
for each $n$. On the other hand, since $T$ is topologically mixing and Markov, there is a positive integer $\ell=\ell(j,j')$ such that $T^{\ell}(j')\supset j$. It implies for each $n$,

\begin{equation}\label{e:20}m^{\P}_{jj}(n)\le m^{\Q}_{j'j'}(n+k).\end{equation}

Using \eqref{e:19} and \eqref{e:20}, we can write,
$$\sum_{n\ge 0}m^{\Q}_{j'j'}(n)\lambda^{-n}\le \sum_{n\ge 0}m^{\P}_{jj}(n)\lambda^{-n}\le \lambda^k\sum_{n\ge 0}m^{\Q}_{j'j'}(n+k)\lambda^{-n-k}.$$
Hence by the third row of Table 1, $M^{\P}$ is recurrent if and only if $M^{\Q}$ is recurrent. Again from \eqref{e:19} and \eqref{e:20} we can see that
$\lim_nm^{\P}_{jj}(n)\lambda^{-n}$ is positive if and only if $\lim_nm^{\Q}_{j'j'}(n)\lambda^{-n}$ is positive and the fourth row of Table 1 for $R=\lambda^{-1}$ can be applied.

In order to distinguish weak, resp. strong recurrence, for a $\P'\subset\P$ let $\Q'\subset Q$ be such that
\begin{equation}\label{e:14}\Q'=\{j'\in \Q\colon~j'\subset j\text{ for some }j\in\P'\}.\end{equation}
Using \eqref{e:19} and \eqref{e:20} again we can see that the Perron values of the irreducible aperiodic matrices $M^{\P'}$ and $M^{\Q'}$ coincide hence the Gurevich entropies $h(M^{\P})$, $h(M^{\P'})$ are equal if and only if it the case for $h(M^{\Q})$, $h(M^{\Q'})$; now Theorem \ref{t:1}(ii),(iii) applies.

Second, if $P\nsubseteq Q$ and $Q\nsubseteq P$, let us consider the partition $\R$, where any element of $\R$ equals the closure of a connected component of the set $I\setminus (P\cap Q)$. The reader can easily verify that $\R$ is a Markov partition for $T$. By the previous, the pairs of matrices  $M^{\P}$, $M^{\R}$, resp. $M^{\R}$, $M^{\Q}$ belong to the same class of the Vere-Jones classification. So it is true also for the pair $M^{\P}$, $M^{\Q}$.
\end{proof}

\begin{Remark}Let $(T,\P,M)\in\CPM^*$. Applying Proposition~\ref{p:10} in what follows we will call $T$ transient, null recurrent, weakly recurrent or strongly recurrent respectively if it is the case for its transition matrix $M$. The last three, resp.\ two possibilities will occasionally be summarized by '$T$ is recurrent', resp. '$T$ is positive recurrent'.
It is well known that if $T$ is piecewise monotone then it is strongly recurrent \cite[Theorem 0.16]{Wal82}.
\end{Remark}

\section{Linearizability}\label{s:6}
In this section we investigate in more details the set of maps from $\CPM$ that are conjugate to maps of constant slope (linearizable, in particular). Relying on Theorem~\ref{t:5}, Theorem~\ref{t:2} and Proposition~\ref{p:10} our main tools will be local and global perturbations of maps from $\CPM$ resulting to maps from $\CPM$. Some examples illustrating the results achieved in this section will be presented in Section~\ref{s:7}.

We start with an easy but rather useful observation. Its second part will play the key role in our evaluation using centralized perturbation - formula \eqref{a:500} and its applications.

\begin{Proposition}\label{p:3}~Let $(T,\P,M)\in\CPM^*$. \begin{itemize}\item[(i)] If $T$ is leo then any $\lambda$-solution of \eqref{e:2} is summable.
\item[(ii)]Any $\lambda$-solution of \eqref{e:2} satisfies
    \begin{equation*}\forall~\varepsilon\in (0,1/2)\colon~\sum_{j\in \P,j\subset (\varepsilon,1-\varepsilon)}v_j<\infty.\end{equation*}
\end{itemize}
\end{Proposition}
\begin{proof}(i)~Since $T$ is leo, for a fixed element $i$ of $\P$,
there is an $n\in\bbn$ such that $T^n(i)=[0,1]$. Then by Proposition~\ref{p:11}(ii), $m_{ij}(n)\ge 1$ for each $j\in \P$.
This implies that any $\lambda$-solution $v=(v_j)_{j\in \P}$ of \eqref{e:2} satisfies
$$
\lambda^nv_i = \sum_{j\in \P}m_{ij}(n) v_j \ge \sum_{j\in \P} v_j,
$$
so $v\in\ell^1(\P)$.

\noindent (ii)~We assume that $T$ is topologically mixing - see Definition~\ref{def:CPMM}. For any fixed element $i \in \P$ there is an $n\in\bbn$ such that $T^n(i)\supset (\varepsilon,1-\varepsilon)$: since $T$ is topologically mixing, there exist positive integers $n_1$ and $n_2$ such that $T^{m_1}(i)\cap  [0,\varepsilon/2)\neq\emptyset$ for every $m_1\ge n_1$, resp. $T^{m_2}(i)\cap  (1-\varepsilon/2,1]\neq\emptyset$ for every $m_2\ge n_2$. This implies that the interval $T^n(i)$ contains $(\varepsilon,1-\varepsilon)$ whenever $n\ge\max\{n_1,n_2\}$ - fix one such $n$. Then $m_{ij}(n)\ge 1$ for any element $j$ of $\P$ such that $j\subset (\varepsilon,1-\varepsilon)$; hence

$$\lambda^nv_i = \sum_{j\in \P}m_{ij}(n) v_j \ge
\sum_{j\in \P,j\subset (\varepsilon,1-\varepsilon)}v_j.$$
for any $\lambda$-solution $v=(v_j)_{j\in \P}$ of \eqref{e:2}.
\end{proof}

The fundamental conclusion regarding linearizability of a map from $\CPM$ provided by the Vere-Jones theory follows.

\begin{Theorem}\label{t:7} If $T\in\CPM$ is leo and recurrent, then $T$ is linearizable.\end{Theorem}
\begin{proof} By assumption there exists a Markov partition $\P$ for $T$ such that the transition matrix $M=M(T)=(m_{ij})_{i,j\in\P}$ is recurrent. In such a case equation \eqref{e:2} has a $\lambda_M$-solution described in Theorem~\ref{t:2}. Since $T$ is leo the $\lambda_M$-solution is summable  by Proposition~\ref{p:3}(i) and the conclusion follows from Theorem~\ref{t:5}.\end{proof}

\begin{Remark}\label{r:7} In Section~\ref{s:7} we present various examples illustrating Theorem~\ref{t:7}. In particular, we show a strongly recurrent non-leo map of an operator type that is not conjugate to any map of constant slope.\end{Remark}

\subsection{Window perturbation}In this subsection we introduce and study two types of perturbations of a map $T$ from $\CPM$: local and global window perturbation.

\subsubsection{Local window perturbation}
\begin{Definition}\label{d:2}
For $S\in\CPM$ with Markov partition $\P$, let $j\in\P$ such that $S_{\vert j}$ is monotone. We say that $T\in\CPM$ is a {\it window perturbation of $S$ on $j$ (of order $k$, $k\in\bbn$)},  if
\begin{itemize}
\item $T$ equals $S$ on $[0,1]\setminus j^{\circ}$
\item there is a nontrivial partition $(j_i)_{i=1}^{2k+1}$ of $j$ such that $T(j_i)=S(j)$  and $T\vert_{j_i}$ is monotone for each $i$.\end{itemize}
\end{Definition}

Notice that due to Definition~\ref{d:2} a window perturbation does not change partition $\P$ (but renders it slack). Using a sufficiently fine Markov partition for $S$, its window perturbation $T$ can be arbitrarily close to $S$ with respect to the supremum norm.

  In the above definition an element of monotonicity of a partition is used. So, for example we can take $\P$ classical (\ie non-slack), or to a given partition $\P'$ and a given maximal interval of monotonicity $i$ of a map we can consider a partition $\P''$ finer than $\P'$ such that $i\in\P''$.

\begin{Proposition}\label{p:14}Let $T\in\CPM$ be a window perturbation of a map $S\in\CPM$. The following is true.
\begin{itemize}
\item[(i)] If $S$ is recurrent then $T$ is strongly recurrent and $R_T<R_S$.
\item[(ii)] If $S$ is transient then $T$ is strongly recurrent for each sufficiently large $k$.
\end{itemize}
\end{Proposition}
\begin{proof}Fix a partition $\P$ for $S$, let $T$ be a window perturbation of $S$ on $j\in\P$. Applying Proposition~\ref{p:10} it is sufficient to specify the Vere-Jones class of $T$ with respect to $\P$. Consider generating functions $F^S(z)=F^S_{jj}(z)=\sum_{n\ge 1}f^S(n)z^n,\text{ resp. }F^{T}(z)=F^{T}_{jj}(z)=\sum_{n\ge 1}f^{T}(n)z^n$, corresponding to $S$, resp. $T$ and with radius of the convergence $\Phi_S=\Phi^S_{jj}$, resp. $\Phi_T=\Phi^T_{jj}$. Notice that

\begin{equation}\label{e:30}\forall~n\in\bbn\colon~f^T(n)=(2k+1)f^S(n),
\end{equation}
hence $\Phi_S=\Phi_T$.\vskip1mm
(i)~If $S$ is recurrent then by Table 1 and \eqref{e:30},

$$\sum_{n\ge 1}f^S(n)R_S^n=1, \qquad \sum_{n\ge 1}f^T(n)R_S^n=2k+1.$$
Then, since $R_S\le \Phi_S=\Phi_T$,

$$\sum_{n\ge 1}f^T(n)R_{T}^n\le 1<2k+1\le\sum_{n\ge 1}f^T(n)(\Phi_T)^n,$$
hence $R_{T}<\Phi_T$ and $T$ is strongly recurrent.

(ii)~If $S$ is transient then by Table 1 and \eqref{e:30},
$$s=\sum_{n\ge 1}f^S(n)R_S^n<1,~\sum_{n\ge 1}f^T(n)R_S^n=(2k+1)s.$$
If for a sufficiently large $k$, $(2k+1)s>1$, necessarily $R_{T}<R_S=\Phi_S=\Phi_T$ and $T$ is strongly recurrent by Table 1.
\end{proof}

Let $M$ be a matrix indexed by the elements of some $\P$ and representing a bounded linear operator $\M$ on the Banach space $\ell^1=\ell^1(\P)$ - see Section~\ref{s:2}. It is well known \cite[p. 264]{aet.80}, \cite[Theorem 3.3]{aet.80} that for $\lambda>r_{\M}$ the formula

\begin{equation}\label{e:40}\left (\frac{1}{\lambda}M_{ij}\left (\frac{1}{\lambda}\right )=\sum_{n\ge 0}m_{ij}(n)/\lambda^{n+1}\right )_{i,j\in\P}
\end{equation}
defines the resolvent operator $R_{\lambda}(\M)\colon~\ell^1(\P)\to\ell^1(\P)$ to  the operator $$\M_{\lambda}=\lambda I-\M.$$

We will repeatedly use this fact when proving our main results. The following theorem implies that in the space of maps from $\CPM$ of operator type an arbitrarily small (with respect to the supremum norm) local change of a map will result to a linearizable map.

\begin{Theorem}\label{t:6}Let $T\in\CPM$ be a window perturbation of order $k$ of a map $S\in\CPM$ of operator type. Then $T$ is linearizable for every sufficiently large $k$.
\end{Theorem}
\begin{proof}We will use the same notation as in the proof of Proposition~\ref{p:14}.

Let us denote $M^{T(k)}=(m^{T(k)}_{ij})_{i,j\in\P}$ the transition matrix of a considered window perturbation $T(k)$ of $S$, let $\lambda_{T(k)}$ be the value ensured for $M^{T(k)}$ by Proposition~\ref{p:2}, put $R_{T(k)}=1/\lambda_{T(k)}$. Since $S$ is of operator type, it is also the case for each $T(k)$. Using Proposition~\ref{p:14} and Theorem~\ref{t:2} we obtain that for some $k_0$ the perturbation $T(k_0)$ is recurrent and equation \eqref{e:2} is $\lambda_{T(k_0)}$-solvable:

\begin{equation*}\label{e:34}\forall~i\in \P\colon~\sum_{\ell\in \P}m^{T(k_0)}_{i\ell}~F^{T(k_0)}_{\ell j}(R_{T(k_0)})=\lambda_{T(k_0)}F^{T(k_0)}_{ij}(R_{T(k_0)}),\end{equation*}
 where $F^{T(k_0)}_{\ell j}(z)=\sum_{n\ge 1}f^{T(k_0)}_{\ell j}(n)z^n$, $\ell\in\P$.
Since by \eqref{e:30} for each $k$,
$$(2k+1)\sum_{n\ge 1}f^S_{jj}(n)R_{T(k)}^n=\sum_{n\ge 1}f^{T(k)}_{jj}(n)R_{T(k)}^n=1,$$
we can deduce that $(R_{T(k)})_{k\ge 1}$ is decreasing and

\begin{equation}\label{e:33}
\lim_kR_{T(k)}=0.
\end{equation}
By our definition of a window perturbation,  for each $i\in\P\setminus\{j\}$,

\begin{equation}\label{e:32}
\forall~\text{ order }k~\forall~n\in\bbn\colon~f^{T(k)}_{ij}(n)=f^S_{ij}(n).
\end{equation}

Denote $r_{k_0}$ the spectral radius of the operator $\M\colon~\ell^1\to\ell^1$ represented by the matrix $M=M^{T(k_0)}$. Using \eqref{e:33} we can consider a $k>k_0$ for which $\lambda_{T(k)}>r_{k_0}$. Then, since the resolvent operator $(\lambda-\M)^{-1}$ represented by the matrix \eqref{e:40} is defined well for each real $\lambda>r_{k_0}$ as a bounded operator on $\ell^1$ \cite[p. 264]{aet.80}, we obtain from \eqref{e:32}, Remark~\ref{r:1} and \eqref{e:36}

\begin{equation*}\label{e:35}
\sum_{i\in\P}F^{T(k)}_{ij}(R_{T(k)})=1+\sum_{i\in\P,~i\neq j}F^{T(k_0)}_{ij}(R_{T(k)})\le \sum_{i\in\P}M_{ij}(R_{T(k)})<\infty;
\end{equation*}
now since $T(k)$ is recurrent, Theorem~\ref{t:2} and Theorem~\ref{t:5} can be applied.
\end{proof}

Let $(T,\P,M)\in\CPM^*$. For any pair $i,j\in\P$ we define the number
$$
n(i,j)=\min\{n\in\bbn\colon~m_{ij}(n)\neq 0\}.
$$
In the corresponding strongly connected directed graph $G=G(M)$, $n(i,j)$ is the length of the shortest path from $i$ to $j$. In particular, such a path
contains neither $i$ nor $j$ inside, so at the same time
$$
\ell_{ij}(n(i,j))\neq 0,~ f_{ij}(n(i,j))\neq 0
$$
and $\ell_{ij}(n)=f_{ij}(n)=0$ for every $n<n(i,j)$. Since
$$
\frac{n(j',i)-n(j',j)}{n(i,j')+n(j',j)}\le \frac{n(j,i)}{n(i,j)}
$$
for every pair $j,j'\in\P$, the suprema
\begin{equation}\label{e:42}
S(j,\P)\colon=\sup_{i\in\P}\frac{n(j,i)}{n(i,j)},~j\in\P
\end{equation}
are either all finite or all infinite. Moreover, we have the following.

\begin{Proposition}\label{p:13}Let $T\in\CPM$ with two Markov partitions $\P$, resp. $\Q$. Then $S(k,\P)$ is finite for some $k\in\P$ if and only if $S(k',\Q)$ is finite for some $k'\in\Q$.
\end{Proposition}

\begin{proof} Let $P=[0,1]\setminus\bigcup_{j\in\P}j^{\circ}$ and $Q=[0,1]\setminus\bigcup_{j\in\Q}j^{\circ}$. First, let us assume that $P\subset Q$. Fix two elements $j\in\P$, $j'\in\Q$ such that $j'\subset j$. Since the map $T$ is topologically mixing, there exists a positive integer $m$ for which $T^mj'\supset j$. For an $i\in\P$ and an $i'\in\Q$ satisfying $i'\subset i$ we obtain $n(i',j')\ge n(i,j)$ and $n(j',i')\le n(j,i)+m$; hence

\begin{equation}\label{e:53}
(\forall~i\in\P)(\forall~i'\in\Q,~i'\subset i)\colon~\frac{n(j',i')}{n(i',j')}\le \frac{n(j,i)+m}{n(i,j)}.
\end{equation}
Inequality \eqref{e:53} together with property \eqref{e:42} show that if $S(k,\P)$ is finite for some $k\in\P$ then $S(k',\Q)$ is finite for some $k'\in\Q$.

On the other hand, there has to exist an $i''\in\Q$, $i''\subset i$ such that $T^{n(i,j)}i''\supset j'$, \ie $n(i,j)\ge n(i'',j')$. Since also $n(j,i)\le n(j',i'')$, we can write for $i''\in\Q$

\begin{equation}\label{e:54}
(\forall~i\in\P)(\exists~i''\in\Q,~i''\subset i)\colon~ \frac{n(j,i)+m}{n(i,j)}\le  \frac{n(j',i'')+m}{n(i'',j')}.
\end{equation}
inequality \eqref{e:54} together with property \eqref{e:42} show that if $S(k',\Q)$ is finite for some $k'\in\Q$ then $S(k,\P)$ is finite for some $k\in\P$.

If $P\nsubseteq Q$ and $Q\nsubseteq P$, we can consider the partition for $T$ $$\R=\P\vee\Q=\{i\cap i'\colon~i\in\P,~i'\in\Q\}.$$ Clearly, $R=[0,1]\setminus\bigcup_{j\in\R}j^{\circ}=P\cup Q$ and we can use the above arguments for the pairs $\R,\P$ and $\R,\Q$ hence the conclusion for the pair $\P,\Q$ follows.
\end{proof}

So, in \eqref{e:42}, for fixed $(T,\P,M)\in\CPM^*$ and $j\in\P$, we compare the shortest path from $j$ to $i$ (numerator) to the shortest path from $i$ to $j$ (denominator) and take the supremum with respect to $i$. For example, for our map from Subsection~\ref{ss:2} the values \eqref{e:42} are equal to $1$, when $T$ is leo, \eqref{e:42} is finite for every $j\in\P$. Theorem~\ref{t:6} explains the role of a window perturbation in case of maps of operator type. In Theorem~\ref{t:9} we obtain an analogous statement for maps of non-operator type under the assumption that the quantities in \eqref{e:42} are finite.

\begin{Theorem}\label{t:9}
Let $S\in\CPM$ with a Markov partition $\Q$ and such that the supremum
in \eqref{e:42} is finite for some $j'\in\Q$. Let $T\in\CPM$ be a window perturbation of order $k$ of $S$. Then $T$ is linearizable for every sufficiently large $k$.
\end{Theorem}

\begin{proof}Fix a partition $\P$ for $S$ and $j\in\P$. A perturbation of $S$ on $j$ of order $k\in\bbn$ will be denoted by $T(k)$. By our assumption, Proposition~\ref{p:13} and \eqref{e:42}, the supremum $S(j,\P)$ is finite. The numbers $n(j,i),n(i,j)$, $i\in\P$, do not depend on any window perturbation on an element of $\P$, because such a perturbation does not change $\P$; we define $V(n)=\{i\in\P\colon~n(i,j)=n\}$,  $c(n)=\max\{n(j,i)\colon~i\in V(n)\}$, $V(n,p)=\{i\in V(n)\colon~n(j,i)=p\}$, $1\le p\le c(n)$. Obviously for every $n$,

\begin{equation}\label{e:43}\frac{c(n)}{n}\le \sup_{i\in\P}\frac{n(j,i)}{n(i,j)}=S(j,\P)<\infty.\end{equation}

To simplify our notation, using Proposition~\ref{p:14} we will assume that $S$ is strongly recurrent, so this is also true for $T(k)$. Similarly as in the proof of Proposition~\ref{p:14} we obtain for each $k$,

\begin{equation}\label{e:45}1/\lambda_{T(k)}=R_{T(k)}<R_S=1/\lambda_S< \Phi_S=\Phi_{T(k)}=1/\lambda.\end{equation}
Moreover, as in \eqref{e:33}, the sequence $(R_{T(k)})_{k\ge 1}$ is decreasing and $\lim_kR_{T(k)}=0$, \ie $\lim_k\lambda_{T(k)}=\infty$.

Let us show that for each sufficiently large $k$ there is a summable $\lambda_{T(k)}$-solution $v=(v_i)_{i\in \P}$ of equation \eqref{e:2}. Using \eqref{e:32}, we can write for any $\varepsilon>0$, sufficiently large $n_0=n_0(\varepsilon)\in\bbn$ and some positive constants $K,K'$,

\begin{align}\label{a:6}
B&:=\sum_{n\ge n_0}\sum_{i\in\P\setminus\{j\}}f^{T(k)}_{ij}(n)R_{T(k)}^n=\sum_{n\ge n_0}\sum_{i\in\P\setminus\{j\}}f^S_{ij}(n)R_{T(k)}^n \\
\nonumber& \le \sum_{n\ge n_0}\sum_{m\ge n}\sum_{p=1}^{c(n)}\sum_{i\in V(n,p)\setminus\{j\}}\!\!\! \!\!\! \ell^S_{ji}(p)f^S_{ij}(m)R_{T(k)}^m\le \sum_{n\ge n_0}\sum_{m\ge n}\sum_{p=1}^{c(n)}f^S_{jj}(p+m)R_{T(k)}^m \\
\nonumber&\le\sum_{n\ge n_0}\sum_{m\ge n}\sum_{p=1}^{c(n)}(\lambda+\varepsilon)^{p+m}R_{T(k)}^m\le
K\cdot\sum_{n\ge n_0}(\lambda+\varepsilon)^{c(n)}\sum_{m\ge n}\left(\frac{\lambda+\varepsilon}{\lambda_{T(k)}}\right)^m  \\
\label{a:7}&\le K'\cdot\sum_{n\ge n_0}\left [\frac{(\lambda+\varepsilon)^{1+\frac{c(n)}{n}}}{\lambda_{T(k)}}\right ]^n.
\end{align}
Since by \eqref{e:45} the value $\lambda$ does not depend on $k$ and $\lim_k\lambda_{T(k)}=\infty$, from \eqref{e:43} follows that

\begin{align}
\label{a:3}&\frac{(\lambda+\varepsilon)^{1+\frac{c(n)}{n}}}{\lambda_{T(k)}}\le \frac{(\lambda+\varepsilon)^{1+S(j,\P)}}{\lambda_{T(k)}}<9/10
\end{align}
for any $k>k_1$.
Clearly the value

$$A=\sum_{n=1}^{n_0-1}\sum_{i\in\P}f^{T(k)}_{ij}(n)R_{T(k)}^n$$
given by a finite number of summands is finite, so taking \eqref{a:6}, \eqref{a:7} and \eqref{a:3} together, using $\sum_{n\ge n_0}f^{T(k)}_{jj}(n)R_{T(k)}^n<F^{T(k)}_{jj}(R_{T(k)})=1$ we obtain
$$\sum_{i\in\P}F^{T(k)}_{ij}(R_{T(k)})=A+(B+1)\le A+1+ K'\cdot\sum_{n\ge n_0}(9/10)^n<\infty$$
whenever $k>k_1$.  This finishes the proof.
\end{proof}

\subsubsection{Global window perturbation}

Let $S$ be from $\CPM$. In this part we will consider a perturbation of $S$ with a Markov partition $\P$ consisting of infinitely many window perturbations on elements of $\P$ (and with independent orders) done due to Definition~\ref{d:2}.

\begin{Definition}\label{d:6}A perturbation $T$ of $S$ on $\P'\subset \P$ will be called {\it centralized} if there is an interval $[a,b]$, $a,b\in (0,1)\setminus\bigcup_{i\in\P}i^{\circ}$ such that $\bigcup\P'\subset [a,b]$.
\end{Definition}

For technical reasons we consider also an {\it empty perturbation} ($T=S$) as centralized.

Let $T$ be a global (centralized) perturbation of $S$ on $\P'\subsetneq \P$, denote $Q=\P\setminus\P'$.
We can write for $j\in\P'$
\begin{align}\label{a:500}
\sum_{i\in\Q}F^T_{ij}(R_T) &=\sum_{i\in\Q}\sum_{n\ge 1}\sum_{k\in\P'\setminus\{j\}}g^{\P'}_{ik}(n)R_T^nF^T_{kj}(R_T)+
\sum_{i\in\Q}\sum_{n\ge 1}g^{\P'}_{ij}(n)R^n_T \\
\nonumber&=\sum_{k\in\P'\setminus\{j\}}F^T_{kj}(R_T)\sum_{i\in\Q}\sum_{n\ge 1}g^{\P'}_{ik}(n)R_T^n+
\sum_{i\in\Q}\sum_{n\ge 1}g^{\P'}_{ij}(n)R^n_T,
\end{align}
where the coefficients $g^{\P'}_{ij}(n)$ were defined before Remark~\ref{r:1}. We use formula \eqref{a:500} to argue in our proofs.

In the next theorem the perturbation $T$ need not be of an operator type.

\begin{Theorem}\label{t:3}
Let $(S,\P,M)\in\CPM^*$ be recurrent and linearizable. Assume that $T$ is a recurrent centralized perturbation of $S$ on $\P'$. If there are finitely many elements of $\P'$ that are $S$-covered by elements of $\P\setminus\P'$, then $T$ is linearizable.
\end{Theorem}
\begin{proof}Let $k_1,\dots,k_m$ be all elements of $\P'$ that are $S$-covered by elements of $\Q=\P\setminus\P'$.  Then
\begin{align}\label{a:30}&\forall~k\in\P'\colon~\sum_{i\in\Q}\sum_{n\ge 1}g^{\P'}_{ik}(n)R_T^n\le\sum_{i\in\Q}\sum_{n\ge 1}g^{\P'}_{ik}(n)R_S^n\\
\nonumber &\le \max_{1\le\ell\le m}\sum_{i\in\Q}\sum_{n\ge 1}g^{\P'}_{ik_{\ell}}(n)R_S^n\le K :=\max_{1\le\ell\le m}\sum_{i\in\P}F^S_{ik_{\ell}}(R_S)<\infty.
\end{align}
Here, the last inequality follows from our assumption that the map $S$ is recurrent and linearizable together with Theorem~\ref{t:2} and Theorem~\ref{t:5}.
Using \eqref{a:500}, \eqref{a:30} and Proposition~\ref{p:3}(ii) we obtain
\begin{eqnarray*} \sum_{i\in\P}F^T_{ij}(R_T) &=&
\sum_{i\in\P'}F^T_{ij}(R_T)+\sum_{i\in\Q}F^T_{ij}(R_T) \\
\nonumber&\le&  \sum_{i\in\P'}F^T_{ij}(R_T)+K\cdot\left(1+
\sum_{k\in\P'\setminus\{j\}}F^T_{kj}(R_T)\right )<\infty.
\end{eqnarray*}
So by Theorem~\ref{t:2} and Theorem~\ref{t:5} the map $T$ is linearizable.
\end{proof}

In the next theorem the perturbation $T$ need not be of operator type.

\begin{Theorem}\label{t:10}Let $S\in\CPM$ be of operator type. If the transition matrix $M=M(S)$ represents an operator $\M$ of the spectral radius $\lambda_S$ then any centralized recurrent perturbation $T$ of $S$ such that $\htop(T)>\htop(S)$ is linearizable. The entropy assumption is always satisfied when $S$ is recurrent.
\end{Theorem}
\begin{proof}Let $T$ be a centralized perturbation of $S$ on $\P'\subset \P$, denote $Q=\P\setminus\P'$. From Proposition~\ref{p:7} and our assumption on the topological entropy of $S$ and $T$ we obtain $1/\lambda_T=R_T<R_S=1/\lambda_S$.
We can write for $j\in\P'$

\begin{align}\label{a:5}
\sum_{i\in\Q}F^T_{ij}(R_T)&=\sum_{i\in\Q}\sum_{n\ge 1}\sum_{k\in\P'\setminus\{j\}}g^{\P'}_{ik}(n)R_T^nF^T_{kj}(R_T)+
\sum_{i\in\Q}\sum_{n\ge 1}g^{\P'}_{ij}(n)R^n_T \\
\nonumber&=\sum_{k\in\P'\setminus\{j\}}F^T_{kj}(R_T)\sum_{i\in\Q}\sum_{n\ge 1}g^{\P'}_{ik}(n)R_T^n+
\sum_{i\in\Q}\sum_{n\ge 1}g^{\P'}_{ij}(n)R^n_T \\
\label{a:4}&\le \sum_{k\in\P'\setminus\{j\}}\left (\sum_{i\in\Q}F^S_{ik}(R_T)\right )F^T_{kj}(R_T)
+ \sum_{i\in\Q}F^S_{ij}(R_T)=V,
\end{align}
where the last inequality follows from the fact that $g^{\P'}_{ik}(n)\le f_{ik}(n)$ for each $k\in\P'$ and $n\in\bbn$ - for the definition of $g^{\P'}_{ik}(n)$, see before Remark~\ref{r:1}. By our assumption, formula \eqref{e:40} represents the resolvent operator $R_{\lambda}(\M)$ for every $\lambda>\lambda_S$. In particular, $R_{\lambda_T}(\M)$ is a bounded operator on $\ell^1(\P)$ \cite[p. 264]{aet.80}, hence with the help of Remark~\ref{r:1} we obtain

\begin{equation*}\label{e:46}
\forall~k\in\P\colon~\sum_{i\in\Q}F^S_{ik}(R_T)<\sum_{i\in\P}F^S_{ik}(R_T)<
\sum_{i\in\P}M^S_{ik}(R_T)\le\lambda_T\| R_{\lambda_T}(\M)\|
\end{equation*}
and \eqref{a:5}, \eqref{a:4} can be rewritten as
\begin{align}
\label{a:8}\sum_{i\in\P}F^T_{ij}(R_T)&\le\sum_{i\in\P'}F^T_{ij}(R_T)+V \\
\label{a:9}&\le \sum_{i\in\P'}F^T_{ij}(R_T)+\lambda_T\| R_{\lambda_T}(\M)\|\left (1+\sum_{k\in\P'\setminus\{j\}}F^T_{kj}(R_T)\right )<\infty,
\end{align}
because $\sum_{k\in\P'}F^T_{kj}(R_T)<\infty$ for topologically mixing $T$ by Proposition~\ref{p:3}(ii) and Theorem~\ref{t:2}. The conclusion follows from Theorem~\ref{t:2} and \eqref{a:8}, \eqref{a:9}. It was shown in Proposition~\ref{p:14}(i) that for a recurrent $S$ we always have $\htop(T)>\htop(S)$.
\end{proof}

 In order to apply Theorem~\ref{t:10} let us consider any map $R\in\CPM$ of operator type, fix $\varepsilon>0$. By Theorem~\ref{t:6} there is a strongly recurrent linearizable map $S$ of operator type for which $\| R-S\|<\varepsilon$. Similarly as in \eqref{e:33} we can conclude that the transition matrix of $S$ satisfies the assumption of Theorem~\ref{t:10}. By that theorem, any centralized perturbation $T$ (operator/non-operator) of $S$ is linearizable (such a centralized perturbation $T$ can be taken to satisfy $\| R-T\|<\varepsilon $).

\section{Examples}\label{s:7}

\subsection{Non-leo maps in the Vere-Jones classes}

For some $a,b\in\bbn$ consider the matrix $M=M(a,b)=(m_{ij})_{i,j\in{\small\mathbb Z}}$ given as
\begin{equation}\label{e:31}
M(a,b) = \begin{pmatrix}
\ddots &  \ddots &\ddots & & & & & & \\
\ddots & a & 0 & b & 0 & & & & \\
& 0 & a & 0 & b & 0 & & & \\
& & 0 & a & 0 & b & 0 & & \\
& & & 0 & a & 0 & b & 0 & \\
 & & & & 0 & a & 0 & b & \ddots \\
& & & & & & \ddots & \ddots & \ddots
\end{pmatrix}
\end{equation}
Clearly $M$ is irreducible but not aperiodic. It has period $2$, so we consider only $m_{ii}(2n)$. Obviously,
$$
m_{ii}(2n)=\binom{2n}{n}a^nb^n.
$$
Using Stirling's formula, we can write

\begin{equation}\label{e:21}
m_{ii}(2n)\sim \frac{2^{2n}}{\pi^{1/2}n^{1/2}}a^nb^n.
\end{equation}
So, $\lambda_M=2\sqrt{ab}=R^{-1}$. At the same time we can see from \eqref{e:21} that
$$\lim_{n\to\infty}m_{ii}(n)R^n=0 \quad \text{ and } \quad \sum_{n\ge 0}m_{ii}(n)R^n=\infty,
$$
so by Table 1, $M(a,b)$ is null recurrent for each pair $a,b\in\bbn$.

In the following statement we describe a class of maps that are not conjugate to any map of constant slope. In particular they are not linearizable. A rich space of such maps (not only Markov) has been studied by different methods in \cite{MiRo14}.

\begin{Proposition}\label{p:20}Let $a,b,k,\ell\in\bbn$, $k$ even and $\ell$ odd, consider the matrix $M(a,b)$ defined in \eqref{e:31}. Then $N=kM(a,b)+\ell E$ is a transition matrix of a non-leo map $T$ from $\CPM$. The map $T$ is null recurrent and it is not conjugate to any map of constant slope. The matrix $N$ represents an operator $\N$ on $\ell^1(\bbz)$ and

\begin{equation}\label{e:1}\lambda_N=2k\sqrt{ab}+\ell.\end{equation}
\end{Proposition}
\begin{proof}Notice that the entries of $N$ away from resp.\ on the diagonal are even, resp.\ odd. Draw a (countably piecewise affine, for example) graph of a map $T$ from $\CPM$ for which $N$ is its transition matrix. Since $M(a,b)$ is null recurrent, the matrix $N$ is also null recurrent by Proposition~\ref{p:19}. Solving the difference equation

\begin{equation}\label{e:44}a~x_{n-1}+b~x_{n+1}=\lambda~x_n, \qquad n\in\bbz,
\end{equation}
one can verify that equation \eqref{e:2} with $M=M(a,b)$ has a $\lambda$-solution if and only if $\lambda\ge \lambda_M=2\sqrt{ab}$ (this follows also from Corollary~\ref{c:1}) and none of these solutions is summable. So by Proposition~\ref{p:19} and Theorem~\ref{t:5}, the map $T$ is not conjugate to any map of constant slope.
\end{proof}

For some $a,b,c\in\bbn$ let $M=M(a,b,c)=(m_{ij})_{i,j\in {\small{\mathbb N}\cup\{0\}}}$ be given by
\begin{equation}\label{e:4}
M(a,b,c) = \begin{pmatrix}
0 &\ c\ &\ 0\ &\ 0\ &\ 0\ & \dots \\
a & 0 & b & 0 & 0 & \dots \\
0 & a & 0 & b & 0 &  \\
0 & 0 & a & 0 & b &  \\
0 & 0 & 0 & a & 0 &  \\
\vdots &  &  & \ddots & \ddots & \ddots \\
\end{pmatrix},
\end{equation}
Again, the matrix $M$ is irreducible but not aperiodic. It has period $2$,
so we consider only the coefficients $f_{00}(2n)$, see Subsection~\ref{ss:15}.
In order to find a $\lambda$-solution for $M$ we can use the difference equation \eqref{e:44} for $n\ge 0$ with the additional conditions $x_0=1$ and $x_1=\lambda/c$. Using Corollary~\ref{c:1} and the direct computation one can show:

\begin{Proposition}\label{p:M}
\begin{itemize}
\item[(a)]
For any choice of $a,b,c\in\bbn$,
$$
f_{00}(2n)=c~b^{n-1}a^n\frac{1}{n}\binom{2n-2}{n-1}\sim \frac{c~b^{n-1}a^n4^{n-1}}{\pi^{1/2}n(n-1)^{1/2}},
$$
so that $\Phi_{ii}^{-1} = (2\sqrt{ab})^{-1}$.
\item[(b)] If $2b > c$ then
$\lambda_M = 2\sqrt{ab}$ and $M$ is transient.
There is a summable $\lambda_M$-solution for $M$ if and only if $a<b$.
\item[(c)] If $2b = c$ then
$\lambda_M = 2\sqrt{ab}$ and $M$ is null recurrent.
There is a summable $\lambda_M$-solution for $M$ if and only if $a<b$.
\item[(d)] If $2b < c$ then
$\lambda_M=c\sqrt{a/(c-b)}>2\sqrt{ab}$, and $M$ is strongly recurrent.
There is a summable $\lambda_M$-solution for $M$ if and only if $a+b<c$.
\end{itemize}
\end{Proposition}

\iffalse
\begin{itemize}
\item[(a)] $\lambda_M=2\sqrt{ab}$ if and only if $2b\ge c$, if $2b<c$ then $\lambda_M=c\sqrt{a/(c-b)}>2\sqrt{ab}$.
\item[(b)] For any choice of $a,b,c\in\bbn$, $$f_{00}(2n)=c~b^{n-1}a^n\frac{1}{n}\binom{2n-2}{n-1}\sim \frac{c~b^{n-1}a^n4^{n-1}}{\pi^{1/2}n(n-1)^{1/2}},$$
    hence $\lim_{n\to\infty}[f_{00}(2n)]^{1/2n}=2\sqrt{ab}$. The last equality together with (a) show that $M$ is strongly recurrent if and only if $2b<c$, see Table 1.

        \item[(c)] If $2b\ge c$ then there is a summable $\lambda_M$-solution for $M$ if and only if $a<b$.
            \item[(d)] If $2b< c$ then there is a summable $\lambda_M$-solution for $M$ if and only if $a+b<c$.
            \item[(e)] If $2b>c$ then $M$ is transient; if $2b=c$ then $M$ is
null-recurrent;  $2b<c$ then $M$ is strongly recurrent.
\end{itemize}
\fi
Using Propositions ~\ref{p:19} and~\ref{p:M} we can conclude.

\begin{Proposition}\label{p:16}Let $a,b,c\in\bbn$. The following hold:
\begin{itemize}
\item[(i)] The matrix $K=2M(a,b,c)+E$ is a transition matrix of a strongly recurrent non-leo map $T\in\CPM$ if and only if $2b<c$. The map $T$ is not linearizable for $a+b\ge c$.
    \item[(ii)] The matrix $L=2M(a,b,b)+E$ is a transition matrix of a transient non-leo map $T\in\CPM$. The map $T$ is linearizable if $a<b$.
\end{itemize}
\end{Proposition}
\begin{proof} Clearly $K$ and $L$ are transition matrices of non-leo maps from $\CPM$. The property (i), resp. (ii) follows from the above properties (a),(b),(d), resp. (a),(b),(c),(e).
\end{proof}

\subsection{Leo maps in the Vere-Jones classes}\label{ss:1}

We have shown in Section 5 that the subset of maps from $\CPM$ that {\it are linearizable} is sufficiently rich in the case of {\it non-leo maps}  of operator/non-operator type. In order to refine the whole picture, in this paragraph we show how to detect interesting {\it leo} maps of operator/non-operator type. In the next two collections of examples we will use a simple countably infinite Markov partition for the full tent map and test various possibilities of its global window perturbations.

\subsubsection{Perturbations of the full tent map of operator type}\label{ss:11} For the full tent map $S(x)=1-\vert 1-2x\vert$, $x\in [0,1]$, consider the Markov partition
$$\P=\{i_{n}=[1/2^{n+1},1/2^{n}]\colon~n=0,1,\dots\}.
$$ We will
study several global window perturbations of $S$ of the following general form: let $a=(a_n)_{n\ge 1}$ be a sequence of odd positive integers and consider a global window perturbation $T^a$ of $S$ such that

 \begin{itemize}\item the window perturbation on $i_n$ is of order $(a_n-1)/2$ (\ie if $a_n=1$ we do not perturb $S$ on $i_n$).
   \end{itemize}

   Then using the notation of Section~\ref{s:4} and Remark~\ref{r:1} we can consider generating functions $F(z)=F^a(z)=F^a_{00}(z)=\sum_{n\ge 1}f^a_{00}(n)z^n$ corresponding to the element $i_0$: $f^a_{00}(n)=f(n)$ for each $n$. One can easily verify that
\begin{equation}\label{e:16}f(1)=1,~f(n)=a_1\cdots a_{n-1}, n\ge 2.\end{equation}
   With the help of Proposition~\ref{p:7} we denote $\lambda=\lambda_a$, resp. $\Phi=\Phi_a$ the topological entropy of $T=T^a$, resp. radius of convergence of $F^a(z)$; also we put $R=R_a=1/\lambda_a$.

 \vskip1mm
\noindent {\it Strongly recurrent:} First of all, consider the set $A(\ell)=\{1,\dots,\ell\}$ and the choice
\begin{equation*}\label{c:6}a_n(0)=\begin{cases}
1,~n\in A(\ell),\\
3,~n\notin A(\ell).
\end{cases}
\end{equation*}
Then by \eqref{e:16}, $f(n)=1$ for $n\in A(\ell)$ and $f(n)=3^{n-\ell-1}$ for each $n\ge \ell+1$,  hence
\begin{equation}\label{e:29}
\lim_{n\to\infty}[f(n)]^{1/n}=3, \quad \Phi=1/3, \quad
\sum_{n\ge 1}f(n)\Phi^n=\infty.
\end{equation}
Therefore by Table 1, $R<\Phi$, \ie $\htop(T)=\log\lambda\in (\log3,\log4)$ - for the upper bound, see \cite{ALM00}.  This implies that the map  $T_{a(0)}$ is strongly recurrent hence by Theorems~\ref{t:2} and \ref{t:6} also linearizable for any $\ell$.

\vskip1mm
\noindent {\it Transient:} Denoting $B(1)=\{1,2,3,4\}\cup \bigcup_{k\ge 2}\{3^k+1,3^k+2\}$ let us define

\begin{equation}\label{c:2}a_n(1)=\begin{cases}
1,~n\in B(1),\\
3,~n\notin B(1).
\end{cases}
\end{equation}

From \eqref{e:16} we obtain $\lim_{n\to\infty}[f(n)]^{1/n}=3$, \ie $\Phi=1/3$. Moreover, by direct computation we can verify that
\begin{equation}\label{e:28}
\sum_{n\ge 1}f(n)\Phi^n<1\text{ hence also }\sum_{n\ge 1}f(n)R^n<1
\end{equation}
since always $R\le \Phi$. It means that the map $T_{a(1)}$ define by the choice
\eqref{c:2} is transient and by Table 1 from Section 4 in fact $R=\Phi$, \ie Proposition~\ref{p:7} implies $\htop(T)=\log 3$. If we consider in \eqref{c:2} any set $B'(1)\supset B(1)$ such that the inequalities \eqref{e:28} are still satisfied, the same is true for a resulting perturbation $T'$.

\begin{Remark} Misiurewicz and Roth \cite{MiRo16}
observed that the map $T_{a(1)}$ is not conjugate to any map of constant slope. It can be shown that for each choice of a sequence $a=(a_n)_{n\ge 1}$ such that the corresponding $T$ has finite topological entropy the following dichotomy is true: either $T$ is recurrent and then equation \eqref{e:2} has no $\lambda$-solution for $\lambda>e^{h(T)}$, or $T$ is transient and then equation \eqref{e:2} does not have any $\lambda$-solution.\end{Remark}

\vskip1mm
\noindent {\it Null recurrent:} The choice \eqref{c:2} was proposed to satisfy $f(n)\sim 3^n/n^2$. Using this fact and \eqref{e:28} we obtain ($R=1/3$)

$$
\sum_{n\ge 1}f(n)R^n<1 \quad \text{ and } \quad \sum_{n\ge 1}nf(n)R^n=\infty.
$$
Let us define inductively a new set $B(2)\subset B(1)$ as follows: put $n_0=0$ and $B(2,0)=B(1)$; assuming that for some $k\in\bbn\cup\{0\}$ we have already defined $n_k$ and $B(2,k)\subset B(1)$, to obtain $B(2,k+1)$ we omit from $B(2,k)$ the least number - denoted $n_{k+1}$- such that the choice

\begin{equation*}\label{c:3}a_{n,k+1}=\begin{cases}
1, &n\in B(2,k+1),\\
3, &n\notin B(2,k+1)
\end{cases}
\end{equation*}
still gives $\sum_{n\ge 1}f(n)R^n<1$ for corresponding window perturbation of $T_{B(2,k)}$. Clearly $n_k<n_{k+1}$ for each $k$. Let $B(2)=\bigcap_{k\ge 0}B(2,k)$ and consider the global perturbation of $S$ corresponding to $a(2)=(a_n(2))_{n\ge 1}$ given by formula \eqref{c:2} with $B(1)$ replaced by $B(2)$. The set $B(2)$ contains infinitely many units (by \eqref{e:29} any choice $A(\ell)$ gives $\sum_{n\ge 1}f(A(\ell);n)R^n=\infty$). Moreover, our definition of $B(2)$ implies $R=L=1/3$,

$$\sum_{n\ge 1}f(n)R^n=1\quad \text{ and }\quad \sum_{n\ge 1}nf(n)R^n=\infty.$$
So, the corresponding perturbation $T_{B(2)}$ of $S$ is null recurrent hence by Theorem~\ref{t:6} it {\it is linearizable}. By Proposition~\ref{p:7}, $\htop(T_{B(2)})=\log3$.

\subsubsection{One more collection of  perturbations of the full tent map}\label{ss:12}

Expanding on the example of Ruette \cite[Example 2.9]{Rue03}
(see also \cite[page 1800]{Pr64}), we have the following construction.
Let $(a_n)_{n \ge 0}$ be a non-negative integer sequence with $a_0=0$,
let $\lambda > 1$ a slope determined below in \eqref{eq:lambda},
and let $i_n = [\lambda^{-n}, \lambda^{-(n+1)}]$, $n \ge 0$, be
adjacent intervals converging to $0$.
Also let $j_n$, $n \ge 1$, be adjacent intervals of
length $\lambda^{-(n+1)}(1-\lambda^{-1})(1+2a_n)$ converging to $1$
and such that $\lambda^{-2}(2\lambda-1)$ is the left boundary point of $j_1$.

\begin{figure}[ht]
\unitlength=5mm
\begin{picture}(14,13)(-1,0) \let\ts\textstyle
\put(0,0){\line(1,0){12}}\put(0,0){\line(0,1){12}}
\put(0,12){\line(1,0){12}} \put(12,0){\line(0,1){12}}
\put(0,0){\line(1,1){12}}
\put(0,2){\line(1,0){12}} \put(-1.2,1.9){\tiny$\lambda^{-1}$}
 \put(-0.5,7){\tiny$I_0$}
\put(0,1.2){\line(1,0){12}} \put(-1.2,1.1){\tiny$\lambda^{-2}$}
 \put(-0.5,1.6){\tiny$I_1$}
\put(0,0.8){\line(1,0){12}} \put(-1.2,0.6){\tiny$\lambda^{-3}$}
\thicklines
\put(0,0){\line(1,6){2}}
\put(2,12){\line(1,-6){1.8}}
\put(3.8,1.2){\line(1,6){0.13}}
\put(3.98, 2){\line(1,-6){0.13}}
\put(4.16,1.2){\line(1,6){0.13}}
\put(4.34, 2){\line(1,-6){0.13}}
\put(4.52,1.2){\line(1,6){0.13}}
\put(4.70, 2){\line(1,-6){0.2}}
\put(4.90,0.8){\line(1,6){0.03}}
\put(5.2,0.85){$\cdots$} \put(7.2,0.4){$\cdots$}
\end{picture}
\caption{The map $T\in\CPM$.}
\label{fig:map1}
\end{figure}
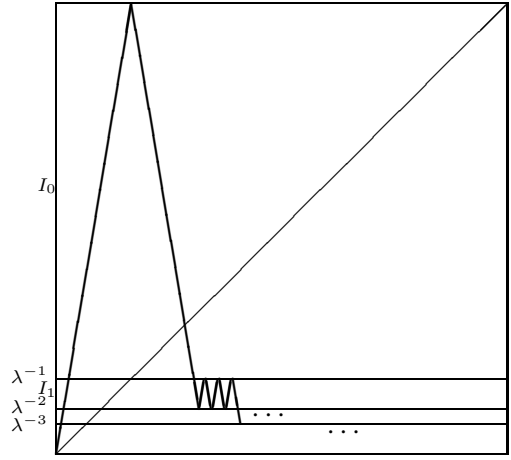

We define the interval map $T:[0,1] \to [0,1]$ with slope $\pm \lambda$
by
\begin{eqnarray*}
T(x) = \begin{cases}
\lambda x & \text{ if } x \in [0, \lambda^{-1}], \\
2-\lambda x & \text{ if } x \in [\lambda^{-1}, \lambda^{-2}(2\lambda-1)], \\
\text{composed of } 1+2a_n & \text{ branches of slope } \mp \lambda \text{ alternatively } \\
\text{mapping into } i_n & \text{ if } x \in j_n,~ n \ge 1.
\end{cases}
\end{eqnarray*}
To make sure that $\lim_{x \to 1} f(x) = 0$, we need $\lambda \in (0, \infty)$
to satisfy
\begin{equation}\label{eq:lambda}
\lambda = 2 + \sum_{n \ge 1} 2a_n (1-\lambda^{-1}) \lambda^{-n}.
\end{equation}
So, any sequence $(a_n)_{n\ge 0}$ such that \eqref{eq:lambda} has a positive finite solution $\lambda$ leads to the {\it linearizable} map $T\in\CPM_{\lambda}$. One can easily see that  $\P=\{i_n\}_{n\ge 0}$ is a Markov (slack) partition for $T$ as defined in Section~\ref{s:2}.

Applying Proposition~\ref{p:12} we associate to $\P$ the transition
matrix
$$
M = M(T) = (m_{ij})_{i,j\in \P} = \begin{pmatrix}
1 & 1+2a_1 & 1+2a_2 & 1+2a_3 & \cdots \\
1 & 0      & 0      & 0      & \dots  \\
0 & 1      & 0      &        &  \\
\vdots & 0 & \ddots & \ddots &
\end{pmatrix}
$$
and also the corresponding strongly connected directed graph $G=G(M)$:
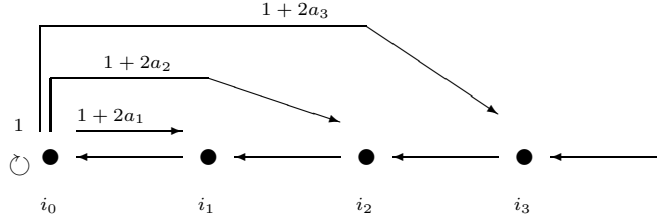
\begin{figure}[ht]
\unitlength=7mm
\begin{picture}(12,4)(0,0) \let\ts\textstyle
\put(1,1){\circle*{0.3}} \put(3.5,1){\vector(-1,0){2}}
\put(0.8,0){\tiny$i_0$}
\put(4,1){\circle*{0.3}} \put(6.5,1){\vector(-1,0){2}}
\put(3.8,0){\tiny$i_1$}
\put(7,1){\circle*{0.3}} \put(9.5,1){\vector(-1,0){2}}
\put(6.8,0){\tiny$i_2$}
\put(10,1){\circle*{0.3}} \put(12.5,1){\vector(-1,0){2}}
\put(9.8,0){\tiny$i_3$}

\put(0.2,0.7){$\circlearrowright$} \put(0.3,1.5){\tiny $1$}
\put(1.5,1.7){\tiny $1+2a_1$}
\put(1.5,1.5){\vector(1,0){2}}
\put(2,2.7){\tiny $1+2a_2$}
\put(1,1.5){\line(0,1){1}}
\put(1,2.5){\line(1,0){3}}\put(4,2.5){\vector(3,-1){2.5}}
\put(5,3.7){\tiny $1+2a_3$}
\put(0.8,1.5){\line(0,1){2}}
\put(0.8,3.5){\line(1,0){6.2}}\put(7,3.5){\vector(3,-2){2.5}}
\end{picture}
\caption{The Markov graph of $T\in\CPM_{\lambda}$;
$1+2a_n$ indicates the number of edges in $G$ from $i_0$ to $i_n$.}
\label{fig:graph1}
\end{figure}

In particular, the number of loops of length $n$ from $i_0$
to itself is $f_{00}(n)=1+2a_{n-1}$.

We use the rome technique from \cite{BGMY} (see also \cite[Section 9.3]{BB})
to compute the entropy of this graph:
it is the leading root of the equation
\begin{equation}\label{eq:entropy1}
z = 1+\sum_{n \ge 1} (1+2a_n) z^{-n}.
\end{equation}
If we divide this equation by $z$, then we get
$$
1= z^{-1} + \sum_{n \ge 1} (1+2a_n) z^{-(n+1)} = \sum_{n\ge 1} f_{00}(n)z^{-n};
$$
from Table 1 follows that the graph $G$ (the matrix $M$, the map $T$) is recurrent for any choice of a sequence $(a_n)_{n \ge 0}$ and corresponding finite $\lambda>0$. Proposition~\ref{p:7} and comparing equations \eqref{eq:lambda} and \eqref{eq:entropy1} we find that $e^{\htop(T)}=\lambda_M=\lambda$.

By Remark~\ref{r:3} the map $T$ is of operator type if and only if $\sup_na_n<\infty$. In this case, by Table 1 and Proposition~\ref{p:7},  $\Phi_{00}=1>1/2\ge R=1/\lambda_M$, so the corresponding map is always strongly recurrent. For the choice  $a_n = a^n$ for some fixed integer $a \ge 2$ the map $T$ is of non-operator type. In this case, $\sum_{n\ge 1}f_{00}(n)a^{-n}=\infty$, so $e^{-\htop(T)}=1/\lambda_M=R<a^{-1}=\Phi_{00}$, hence by Table 1 the map $T$ is still strongly recurrent.

We can also take $a_n = a^{n-\psi_n}$ for some sublinearly growing integer
sequence $(\psi_n)_{n \ge 1}$ chosen such that \eqref{eq:entropy1}
holds for $z = a$, \ie $a = 1+ \sum_{n \ge 1} a^{-n} + 2a^{-\psi_n}$.
In this case, $\Phi_{00} = R$ and $\sum_n f_{00}^{(n)} R^n  = 1$,
and the system is null-recurrent or weakly recurrent (not strongly
recurrent) depending on whether
$\sum_n n a^{-\psi_n}$ is infinite or finite.

\subsubsection{Transient non-operator example from \cite{BT12}}\label{ss:13}
Although up to now all our main results have been formulated and proved in the context of continuous maps, many statements remain true also for countably piecewise monotone Markov interval maps that are countably piecewise continuous.\footnote{This example can be made continuous by replacing each branch with a tent-map of the same height. The $\lambda$ will be twice as large, and the entropy increases accordingly in that case}
We will present a countably piecewise continuous, countably piecewise monotone example adapted from \cite{BT12}, where
it is studied in detail for its thermodynamic properties.

Let $(w_k)_{k \ge 0}$ be a strictly decreasing sequence in $[0,1]$ with
$w_0 = 1$ and $\lim_k w_k = 0$. We will consider the partition $\P=\{p_k\}_{k\in\bbn}$, where the interval map $T$ is designed to be linear increasing on each interval
$p_k = [w_k, w_{k-1})$ for $k \ge 2$, $p_1 = [w_1,w_0]$, $T(p_k) = \bigcup_{i\ge k-1}p_i$ for $k \ge 2$ and $T(p_1) = [0,1]$. With a slight modification of our definition from Section~\ref{s:2}, $\P$ is a Markov partition for $T$ and $T$ is leo. Let $M=M(T)$ be the matrix corresponding to $\P$, see below. In order to have constant slope $\lambda$, we need to solve
the recursive relation

$$
\begin{cases}
w_{k+1} = w_k - w_{k-1}/\lambda & \text{ for } k \ge 1; \\
w_0 = 1, \ w_1 = 1-1/\lambda &
\end{cases}
$$
The characteristic equation $\alpha^2 - \alpha + 1/\lambda = 0$
has real solutions $\alpha_\pm = \frac12(1\pm \sqrt{1-4/\lambda}) \in (0,1)$
whenever $\lambda \ge 4$. We obtain the solution:

$$
w^4_k = 2^{-k} (1+k/2)\quad  \text{ if } \lambda = 4,
$$
and
$$
w^{\lambda}_k = \frac{1-2/\lambda}{2 \sqrt{1-4/\lambda}} \alpha_+^k +
\frac{2\sqrt{1-4/\lambda}-1+2/\lambda}{2 \sqrt{1-4/\lambda}} \alpha_-^k \quad
\text{ if } \lambda > 4.
$$

\begin{figure}[ht]
\unitlength=4mm
\begin{picture}(9,11)(-10,-0.5) \let\ts\textstyle
\put(-21,6){$M = \begin{pmatrix}
1 &\ 1\ &\ 1\ & 1 & 1 & \dots \\
1 & 1 & 1 & 1 & 1 & \dots \\
0 & 1 & 1 & 1 & 1 &  \\
0 & 0 & 1 & 1 & 1 &  \\
0 & 0 & 0 & 1 & 1 &  \\
\vdots &  &  & \ddots & \ddots & \ddots \\
\end{pmatrix}$
}
\put(-19,0){
$S^{\lambda}(x):= \lambda(x-w^{\lambda}_k)$ if $x \in p_k$
}
\thinlines
\put(0,0){\line(1,0){10}}\put(0,10){\line(1,0){10}}
\put(0,0){\line(0,1){10}} \put(10,0){\line(0,1){10}}
\put(0,0){\line(1,1){10}}
\thicklines
\put(7.5,0){\line(1,4){2.5}} \put(8.3,-0.9){$\tiny p_1$}
\put(5.5,0){\line(1,5){2}} \put(6.2,-0.9){$\tiny p_2$}
\put(4,0){\line(1,5){1.5}}  \put(4.5,-0.9){$\tiny p_3$}
\put(2.9,0){\line(1,5){1.1}}  \put(3.0,-0.9){$\tiny p_4$}
\put(2.1,0){\line(1,5){0.8}}  \put(1.2,-0.9){$\ldots$}
\put(1.52,0){\line(1,5){0.58}}
\put(1.095,0){\line(1,5){0.425}}
\put(0.791,0){\line(1,5){0.304}}
\put(0.572,0){\line(1,5){0.219}}
\put(0.4138,0){\line(1,5){0.1582}}
\put(0.2974,0){\line(1,5){0.1164}}
\put(0.21464,0){\line(1,5){0.08276}}
\end{picture}
\caption{Right: The map $T\colon~[0,1] \to [0,1]$. }\label{fig:Flambda}
\end{figure}
It is known that $\htop(T) =\log4$ hence by Proposition~\ref{p:7}, $\lambda_M=4$.
If we remove site $p_1$ (\ie remove the first row and column) from $M$,
the resulting matrix is $M$ again, so strongly connected directed graph $G=G(M)$ contains its copy as a proper subgraph and due to Theorem~\ref{t:1}(i), $M$ and also $T$ is transient.

Writing $v^{\lambda}_k = |p_k| = w^{\lambda}_{k-1}-w^{\lambda}_k$, we have found, in accordance with Theorem~\ref{t:2}, a positive summable $\lambda$-solution of
equation \eqref{e:2} for each $\lambda\ge 4$.  Summarizing, the map $T$ is conjugate to a map of constant slope $\lambda$ whenever $\lambda\ge 4$. $T$ is also {\it linearizable}, since $\lambda=4=\lambda_M=e^{\htop(T)}$.

\subsubsection{Transient non-operator example from \cite{BoSou11}}\label{ss:14}
Let $V=\{v_{i}\}_{i\geq -1}$, $X=\{x_{i}\}_{i\geq 1}$
\noindent $V,X$ converge to $1/2$ and
\noindent $0=v_{-1}=x_0=v_0<x_1<v_1<x_2<v_2<x_3<v_3<\cdots$; the interval map
$T=T(V,X):[0,1]\rightarrow [0,1]$ satisfies
\begin{enumerate}
\item[(a)] $T(v_{2i-1})=1-v_{2i-1}$, $i\geq 1$, $T(v_{2i})=v_{2i}$, $i\geq 0$,
\item[(b)] $T(x_{2i-1})=1-v_{2i-3}$, $i\geq 1$, $T(x_{2i})=v_{2i-2},~i\geq 1$,
\item[(c)] $T_{u,v}=\left\vert\frac{T(u)-T(v)}{u-v}\right\vert>1$ for each interval $[u,v]\subset [x_i,x_{i+1}]$,
\item[(d)] $T(1/2)=1/2$ and $T(t)=T(1-t)$ for each $t\in [1/2,1]$.
\end{enumerate}

Property (c) holds for our $V,X$ since by
(a),(b), we have $T_{x_i,x_{i+1}}>2$ for each $i\ge 0$.

\begin{figure}[ht]
\unitlength=10mm
\begin{center}
\epsfig{file=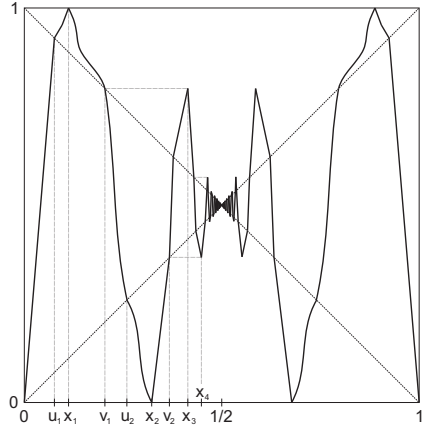,width=5.5cm}
\end{center}
\caption{The leo map $T\in\mathcal{F}\subset \CPM$.}
\end{figure}

Let us denote $\mathcal{F}(V,X)$ the set of all continuous interval maps
fulfilling (a)-(d) for a fixed pair $V,X$ and put  $\mathcal{F}:=\bigcup_{V,X}\mathcal{F}(V,X)$. It was shown in \cite{BoSou11} that

\begin{itemize}
\item $\mathcal{F}$ is a conjugacy class of maps in $\CPM$.
\item strongly connected directed graph $G=G(M)$ contains its copy as a proper subgraph \cite[Theorem 4.5, Fig. 3]{BoSou11}, so due to Theorem~\ref{t:1}(i), $T$ is transient.
\item the common topological entropy equals $\log9$.
\item equation \eqref{e:2} has a positive summable $\lambda$-solution for each $\lambda\ge 9=e^{\htop(T)}$.
\end{itemize}

\begin{figure}[ht]
\unitlength=10mm
\begin{center}
\epsfig{file=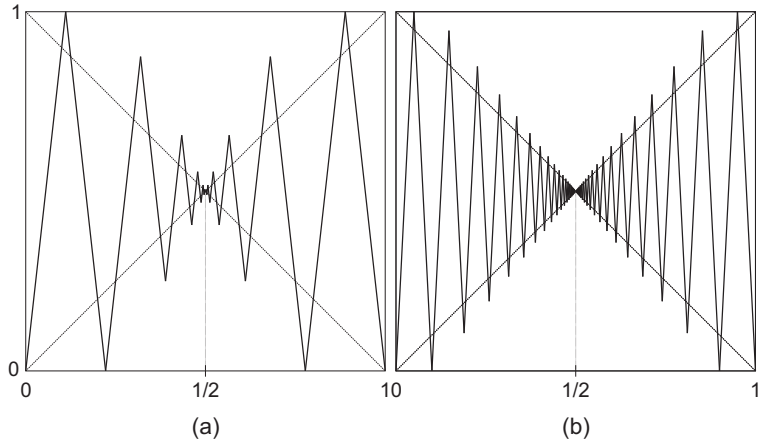,width=10cm}
\end{center}
\caption{$T\in\mathcal{F}$ is conjugate to a map of slope $9$ (a) and
and of slope $20$ (b).}
\end{figure}

We can factor out the left-right symmetry of this map by using the
semiconjugacy $h(x) = 2|x-\frac12|$, and the factor map $\tilde T$
has transition matrix
$$
M = \begin{pmatrix}
4 &\ 4\ &\ 4\ &\ 4\ &\ 4\ & \dots \\
1 & 4 & 4 & 4 & 4 & \dots \\
0 & 1 & 4 & 4 & 4 &  \\
0 & 0 & 1 & 4 & 4 &  \\
0 & 0 & 0 & 1 & 4 &  \\
\vdots &  &  & \ddots & \ddots & \ddots \\
\end{pmatrix}
$$
with similar properties as the previous example. Therefore $T$ is conjugate to a map of constant slope $\lambda$ whenever $\lambda\ge 9$, and also linearizable, since $\lambda=9=\lambda_M=e^{\htop(T)}$.

\subsection{One application of our results}\label{ss:2}Using Proposition~\ref{p:20} let $K=2M(1,1)+E$. We have discussed the fact that $K$ is a transition matrix of a non-leo map $T\in\CPM$ with corresponding Markov partition denoted by $\P$. Clearly by Remark~\ref{r:3} $K$ represents a bounded linear operator - denote it by $\K$ - on $\ell^1(\P)$, so $T$ is of operator type. We can conclude that:
\begin{itemize}
\item[(i)] $\lambda_K=5=e^{\htop(T)}$ - Propositions~\ref{p:7}, \ref{p:20}.
\item[(ii)] $\lambda_K=r_{\K}=\|\K\|$ - Proposition~\ref{p:20}, Section~\ref{s:2}\eqref{e:36}.
        \item[(iii)] $T$ is not conjugate to a map of constant slope (is not linearizable) - Proposition~\ref{p:20}.
    \item[(iv)] $T$ is null recurrent - Proposition~\ref{p:20}.
    \item[(v)] Let $\P'$ be a Markov partition for $T$, denote $K'$ the transition matrix of $T$ with respect to $\P'$ representing a bounded linear operator $\mathcal K'$ on $\ell^1(\P')$.
    Since
 \begin{equation*}\label{e:55}\forall~y\in (0,1)\colon~\#T^{-1}(y)=5,\end{equation*}
we have $\lambda_{K'}=r_{\mathcal K'}=5$ - see (i), Section~\ref{s:2} and Proposition~\ref{p:7}. Then by Theorem~\ref{t:10} any recurrent centralized (operator/non-operator) perturbation of $T$ is linearizable. In particular it is true for any local window perturbation of $T$ on some element of $\P'$ -  Proposition~\ref{p:14}(i).
    \item[(vi)]  Let $\P'$ be a Markov partition for $T$ which equals $\P$ outside of some interval $[a,b]\subset (0,1)$. Let $T'$ be a local window perturbation of $T$ on some element of $\P'$; from the previous paragraph (v) follows that $T'$ is strongly recurrent and linearizable. Consider a centralized (operator/non-operator) perturbation $T''$ of $T'$ on some $\P''\subset\P'$. Then if $T''$ is recurrent it is linearizable by Theorem~\ref{t:3}. Otherwise we can use either Theorem~\ref{t:6} (an operator case) or Theorem~\ref{t:9} (non-operator case, $S(j,\P')$ is finite for $j\in\P'$) to show that a local window perturbation of $T''$ of a sufficiently large order is linearizable.
\end{itemize}

%%%%%%%%%%%%%%%%%%%%%%%%%%%%%%%%%%%%%%%%%%%%%%%%%%%%%

\setcounter{section}{-1} %% Remove this when starting to work on the template.

%%%%%%%%%%%%%%%%%%%%%%%%%%%%%%%%%%%%%%%%%%
\bibliographystyle{mdpi}

\begin{thebibliography}{999}

\bibitem{ALM00} Alsed\'a, Ll.;  Llibre, J.; Misiurewicz, M.\
\emph{Combinatorial dynamics and the entropy in dimension one},
{\em Adv. Ser. in Nonlinear Dynamics} \textbf{5}, 2nd Edition, World Scientific,
Singapore, 2000.

\bibitem{BGMY}
Block, L.; Guckenheimer, J.;  Misiurewicz, M.; Young, L.-S.\
Periodic points and topological entropy of one-dimensional maps,
{\em Lecture Notes in Math.\ } {\bf 1980,} {\em 819} Springer Verlag, Berlin 18-34.

\bibitem{Bo12}  Bobok, J.\
Semiconjugacy to a map of a constant slope,
{\em Studia Math.\ } {\bf 2012,} {\em 208,} 213--228.

\bibitem{Bo03} Bobok, J.\
Strictly ergodic patterns and entropy for interval maps,
{\em Acta Math.\ Univ.\ Comenianae} {\bf 2003,}
{\em LXXII,} 111--118.

\bibitem{BoSou11} Bobok, J.;  Soukenka, M.\
On piecewise affine interval maps with countably many laps,
{\em Discrete and Continuous Dynamical Systems} {\bf 2011,} {\em 31.3,} 753--762.

\bibitem{BB}  Brucks, K.;  Bruin, H.\
Topics from one--dimensional dynamics,
London Mathematical Society, Student Texts {\em 62} Cambridge University Press 2004.

\bibitem{BT12} Bruin, H.; Todd, M.\
Transience and thermodynamic formalism for infinitely branched
interval maps,
{\em Journal of the London Math.\ Soc.\ } {\bf 2012,} {\em 86,} 171--194.

\bibitem{Gur69} Gurevi\v c,  B.\ M.\
Topological entropy for denumerable Markov chains,
{\em Dokl.\ Akad.\ Nauk SSSR} {\bf 1969,} {\em 10,} 911--915.

\bibitem{Chu60} Chung, K.\ L.\
Markov chains with stationary transition probabilities,
Springer, Berlin, 1960.

\bibitem{KH95} Katok,  A.; Hasselblatt, B.\
Introduction to the modern theory of dynamical systems,
Cambridge University Press, Cambridge (1995).

%\bibitem{Ki98} Kitchens, B.\ P.\
%Symbolic dynamics: one-sided, two-sided, and countable state Markov shifts,
%Universitext, Springer-Verlag Berlin Heidelberg New York, 1998.

\bibitem{MiThu88} Milnor, J.; W.\ Thurston, W.\
On iterated maps of the interval,
In: {\em Dynamical Systems, Lecture Notes in Math.\ }
{\em 1342} Springer, Berlin, 1988;  pp.\ 465-563,

\bibitem{Mi79}  Misiurewicz,  M.\
Horseshoes for mappings of an interval,
{\em Bull.\ Acad.\ Pol.\ Sci., S\'er.\ Sci.\ Math.} {\bf 1979,} {\em 27}, 167--169.

\bibitem{MiRa05}  Misiurewicz,  M.; Raith, P.\
Strict inequalities for the entropy of transitive piecewise monotone maps,
{\em Discrete Contin.\ Dyn.\ Syst.} {\bf 2005,} {\em 13}, 451--468.

\bibitem{MiRo14}  Misiurewicz,  M.; Roth, S.\
No semiconjugacy to a map of constant slope,
{\em Ergod.\ Th.\ \& Dynam.\ Sys.} {\bf 2016,} {\em 36,} 875--889.

\bibitem{MiRo16} Misiurewicz,  M.; Roth, S.\
Constant slope maps on the extended real line,
Preprint 2016, arXiv:1603.04198.

\bibitem{MiSl80} Misiurewicz, M.; Szlenk,  W.\
Entropy of piecewise monotone mappings,
{\em Studia Math.} {\bf 1988,} {\em 67(1),} 45-63.

\bibitem{Par66} W.\ Parry,
Symbolic dynamics and transformations of the unit interval,
{\em Trans.\ Amer.\ Math.\ Soc.} {\bf 1966,} {\em 122,} 368--378.

\bibitem{Pr64} Pruitt, W.\
Eigenvalues of non-negative matrices
{\em The Annals of Mathematical Statistics} {\bf 1964,} {\em 35(4),} 1797--1800.

\bibitem{Rue03}  Ruette, S.\
On the Vere-Jones classification and existence of maximal measures for countable topological Markov chains,
{\em Pacific J.\ Math.\ } {\bf 2003,} {\em 209,} 366--380.

\bibitem{Sal88} Salama, I.\
Topological entropy and recurrence of countable chains,
{\em Pacific J.\ Math.\ }{|bf 1988,} {\em 134,} 325--341.
Errata {\bf 1989,} {\em 140(2)}.

\bibitem{aet.80} Taylor, A.\ E.\;  Lay, D.\ C.\
Introduction to Fuctional Analysis,
Robert E. Krieger Publishing Company, 2nd Edition, Malabar, Florida, 1980.

\bibitem{Ver-Jo67} Vere-Jones,  D.\
Ergodic properties of non-negative matrices-I,
{\em Pacific J.\ of Math.\ } {\bf 1967,} {\em 22(2),} 361--385.

%\bibitem{Ver-Jo68} Vere-Jones,  D.\
%Ergodic properties of non-negative matrices-II,
%{\em Pacific J.\ of Math.\ } {\bf 1968,} {\em 26(3),} 601--620.

\bibitem{Wal82}  Walters, P.\
An introduction to ergodic theory,
Springer Verlag (Heidelberg-New York), 1982.

\end{thebibliography}

%=====================================
% References, variant A: internal bibliography
%=====================================
\renewcommand\bibname{References}

%%%%%%%%%%%%%%%%%%%%%%%%%%%%%%%%%%%%%%%%%%
\end{document}